\documentclass[11pt]{article}
\usepackage{smile}

\usepackage{fullpage}
\usepackage{lscape}
\usepackage{bigints}
\usepackage{framed}
\usepackage{mdframed}
\usepackage{enumerate}
\usepackage[T1]{fontenc}
\usepackage{moresize}
\usepackage{bm}
\usepackage{bbm}
\usepackage{dsfont}
\usepackage{amsmath}
\usepackage{amssymb}
\usepackage{amsthm}
\usepackage{amsfonts}
\usepackage{stmaryrd}
\usepackage{array}
\usepackage{mathrsfs}
\usepackage{mathtools} 
\usepackage{extarrows}
\usepackage{stackrel}
\usepackage{relsize,exscale}
\usepackage{scalerel}
\usepackage[nodisplayskipstretch]{setspace}
\usepackage{color}
\usepackage[usenames,dvipsnames]{xcolor}
\usepackage{cancel}
\usepackage{soul}
\usepackage{undertilde}
\usepackage{xfrac}
\usepackage{siunitx}
\usepackage{graphicx}
\usepackage{float}
\usepackage{rotating}
\usepackage{subcaption}
\usepackage{overpic}
\usepackage[all]{xy}
\usepackage{tikz}
\usetikzlibrary{arrows,matrix,positioning,calc,automata,patterns}
\usepackage{booktabs}
\usepackage{dcolumn}
\usepackage{multirow}
\usepackage{diagbox}
\usepackage{tabularx}
\usepackage{verbatim}
\usepackage{listings}
\usepackage[ruled,vlined]{algorithm2e}
\usepackage{fancyvrb}
\usepackage{hyperref}
\usepackage[round]{natbib}
\usepackage{sectsty}

\hypersetup{
    bookmarks=true,         
    unicode=false,          
    pdftoolbar=true,        
    pdfmenubar=true,        
    pdffitwindow=false,     
    pdfstartview={FitH},    
    pdftitle={Efficient estimation of measures of functional dependence},    
    pdfauthor={Author},     
    pdfsubject={Subject},   
    pdfcreator={Creator},   
    pdfproducer={Producer}, 
    pdfkeywords={functional dependence, semiparametric efficiency, cross-fitting}, 
    pdfnewwindow=true,      
    colorlinks=true,        
    linkcolor=blue,         
    citecolor=blue,         
    filecolor=blue,         
    urlcolor=cyan           
}

\usepackage{stackengine}
\stackMath
\newcommand\tenq[2][1]{%
\def\useanchorwidth{T}%
\ifnum#1>1%
\stackunder[0pt]{\tenq[\numexpr#1-1\relax]{#2}}{\!\scriptscriptstyle\thicksim}%
\else%
\stackunder[1pt]{#2}{\!\scriptstyle\thicksim}%
\fi%
}

\makeatletter
\DeclareRobustCommand\widecheck[1]{{\mathpalette\@widecheck{#1}}}
\def\@widecheck#1#2{%
    \setbox\z@\hbox{\m@th$#1#2$}%
    \setbox\tw@\hbox{\m@th$#1%
       \widehat{%
          \vrule\@width\z@\@height\ht\z@
          \vrule\@height\z@\@width\wd\z@}$}%
    \dp\tw@-\ht\z@
    \@tempdima\ht\z@ \advance\@tempdima2\ht\tw@ \divide\@tempdima\thr@@
    \setbox\tw@\hbox{%
       \raise\@tempdima\hbox{\scalebox{1}[-1]{\lower\@tempdima\box
\tw@}}}%
    {\ooalign{\box\tw@ \cr \box\z@}}}
\makeatother

\newcommand{\Rem}{\operatorname{Rem}}

\def\tr{\mathop{\text{tr}}\kern.2ex}

\def\P{{\mathrm P}}

\def\E{{\mathrm E}}
\def\R{{\mathbbm R}}

\def\d{{\mathrm d}}

\newcommand{\sgn}{\mathrm{sgn}}

\newcolumntype{L}[1]{>{\raggedright\let\newline\\\arraybackslash\hspace{0pt}}m{#1}}
\newcolumntype{C}[1]{>{  \centering\let\newline\\\arraybackslash\hspace{0pt}}m{#1}}
\newcolumntype{R}[1]{>{ \raggedleft\let\newline\\\arraybackslash\hspace{0pt}}m{#1}}
\newcolumntype{d}[1]{D{.}{.}{#1}}
\newcolumntype{H}{>{\setbox0=\hbox\bgroup}c<{\egroup}@{}}
\newcolumntype{Z}{>{\setbox0=\hbox\bgroup}c<{\egroup}@{\hspace*{-\tabcolsep}}}
\newcolumntype{b}{X}
\newcolumntype{s}{>{\hsize=.5\hsize}X}

\def\PP{{\mathbb{P}}}
\def\QQ{{\mathbb{Q}}}

\numberwithin{equation}{section}

\newtheorem{theorem}{Theorem}[section]

\newtheorem{proposition}{Proposition}[section]
\newtheorem{assumption}{Assumption}[section]
\newtheorem{corollary}{Corollary}[section]
\newtheorem{innercustomgeneric}{\customgenericname}
\providecommand{\customgenericname}{}
\newcommand{\newcustomtheorem}[2]{%
  \newenvironment{#1}[1]
  {%
   \renewcommand\customgenericname{#2}%
   \renewcommand\theinnercustomgeneric{##1}%
   \innercustomgeneric
  }
  {\endinnercustomgeneric}
}
\newcustomtheorem{customdefinition}{Definition}
\newcustomtheorem{customdefinitions}{Definitions}
\newcustomtheorem{customtheorem}{Theorem}
\newcustomtheorem{customassumption}{Assumption}
\newcustomtheorem{customlemma}{Lemma}
\newcustomtheorem{customexample}{Example}
\theoremstyle{definition}
\newtheorem{definition}{Definition}[section]
\newtheorem{example}{Example}[section]
\newtheorem{remark}{Remark}[section]

\makeatletter
\newcommand{\mylabel}[2]{#2\def\@currentlabel{#2}\label{#1}}
\makeatother

\usepackage{enumitem}

\graphicspath{{./fig3/}}

\allowdisplaybreaks

\begin{document}

\setlength{\abovedisplayskip}{5pt}
\setlength{\belowdisplayskip}{5pt}
\setlength{\abovedisplayshortskip}{5pt}
\setlength{\belowdisplayshortskip}{5pt}
\hypersetup{
  colorlinks,
  breaklinks,
  urlcolor=blue,
  linkcolor=blue,
  pdftitle={Efficient estimation of measures of functional dependence},
  pdfauthor={Holger Dette and Fang Han}
}


\title{\LARGE Efficient estimation of measures of functional dependence}

\author{
Holger Dette\thanks{Faculty of Mathematics, Ruhr University Bochum, Bochum, Germany; email: {\tt holger.dette@rub.de}}
~~~{\rm and}~~
Fang Han\thanks{Department of Statistics, University of Washington, Seattle, WA 98195, USA; e-mail: {\tt fanghan@uw.edu}}
}

\date{\today}

\maketitle

\vspace{-1em}

\begin{abstract}
Measures of functional dependence that characterize both independence and perfect dependence have attracted growing interest. Despite their appealing theoretical properties, existing rank- and graph-based estimators can suffer from substantial bias under general dependence, limiting their usefulness for statistical inference. Moreover, a systematic and practically accessible framework for mitigating such bias while achieving semiparametric efficiency remains largely unavailable. In this paper, we develop a general framework for efficient estimation of measures of functional dependence through sample-splitting and cross-fitting. Our approach enables systematic bias reduction and semiparametrically efficient inference, and thus offers a flexible and theoretically grounded alternative to existing rank- and graph-based methods.
\end{abstract}

\noindent{\bf Keywords}: functional dependence, efficient influence function, semiparametric efficiency, sample-splitting, cross-fitting.

\section{Introduction}

Measuring dependence between random variables $X$ and $Y$ is a fundamental problem in statistics, with applications ranging from independence testing and variable selection to sensitivity analysis and scientific discovery. Classical measures, such as Pearson's correlation coefficient and its rank-based counterparts, provide useful summaries of specific forms of association but may fail to capture more general nonlinear relationships. This limitation has motivated the development of dependence measures that characterize both independence and certain forms of functional dependence; see, for example, the influential discussions in \cite{MR0115203}, \cite{schweizer1981nonparametric}, and \cite{scarsini1984measures}.

Building on the earlier developments of \cite{trutschnig2011strong} and \cite{MR3024030}, this paper focuses on \emph{measures of functional dependence}. Such a measure takes values in the interval $[0,1]$, attains its minimum value $0$ if and only if $X$ and $Y$ are independent, and equals $1$  if and only if  $Y$ is almost surely a measurable function of $X$. These properties provide a natural and interpretable scale for quantifying the extent to which one variable determines another, without imposing a parametric model on their relationship.

A particularly influential development in this direction is the introduction of Chatterjee's rank correlation coefficient and its graph-based extension \citep{chatterjee2020new,azadkia2019simple}, which estimate a measure of functional dependence previously proposed by \cite{MR3024030}. These coefficients have several appealing features: they are based on simple rank and nearest-neighbor (NN) statistics, can be made invariant under strictly monotone transformations through suitable rank transformations \citep{tran2024rank}, and characterize both independence and perfect functional dependence. Moreover, their construction requires neither tuning parameters nor preliminary nonparametric estimation. These developments have stimulated considerable interest in rank- and graph-based measures of functional dependence \citep{lin2022limit,shi2020power,shi2021ac,kroll2024asymptotic,lin2024failure,Dette_Kroll_2025,gao2026limit,chhaibi2026martingale,hormann2026azadkia}, as well as in broader classes of dependence measures sharing similar properties \citep{wiesel2022measuring,strothmann2024rearranged,yang2025coverage,ansari2025ordering,azadkia2025new,ansari2026quantifying,ansari2026dependence}.

Despite their conceptual simplicity and attractive population properties, statistical inference based on existing rank- and graph-based estimators remains challenging. A central difficulty is that these estimators can exhibit substantial bias under general dependence \citep{azadkia2019simple,azadkia2026bias}. In particular, the bias arising from local nearest-neighbor approximations may decay more slowly than the parametric rate, thereby dominating stochastic fluctuations and invalidating conventional confidence intervals. This difficulty becomes especially pronounced when the dimension of  $X$ exceeds three. Recent work by \cite{azadkia2026bias} investigates this issue for Chatterjee's graph-based correlation coefficient and develops a bias-correction procedure that restores root-$n$ consistency and asymptotic normality under suitable conditions. Nevertheless, such corrections are closely tied to the structure of the underlying graph-based estimators, leaving open the question of whether a more systematic and broadly applicable estimation strategy is available.

Beyond bias, another important consideration is statistical efficiency. Even when a dependence estimator is root-$n$ consistent and asymptotically normal, its asymptotic variance need not be optimal \citep{bickel1993efficient,abadie2006large}. From a semiparametric perspective, the natural benchmark is the efficiency bound determined by the efficient influence function (EIF) of the population dependence measure. This perspective is particularly relevant because measures of functional dependence are often nonlinear functionals of conditional distribution functions, which must themselves be estimated from the data. For the  Dette--Siburg--Stoimenov dependence measure, \cite{klein2024asymptotic} recently characterized its EIF. However, translating such efficiency calculations into practically implementable estimators with rigorous theoretical guarantees remains challenging, particularly when conditional distributions are estimated using flexible nonparametric or machine-learning methods. More generally, a unified framework for efficient estimation across different measures of functional dependence remains largely unavailable.

In this paper, we develop such a framework based on sample-splitting, cross-fitting, and one-step estimation. Our approach draws on the double/debiased machine learning methodology of \cite{chernozhukov2018double} and the broader literature on cross-fitted semiparametric estimation \citep{klaassen1987consistent,newey2018cross}. Rather than estimating dependence measures directly through ranks or nearest-neighbor graphs, we exploit their representations as functionals of conditional distribution functions. Specifically, we first estimate the conditional distributions using observations outside each validation fold, and then construct an EIF-based one-step correction evaluated on the held-out observations. Aggregating the resulting estimators across folds yields a flexible procedure that eliminates the leading bias arising from nuisance estimation while retaining the full sample for inference. Under suitable regularity conditions, the proposed estimators are asymptotically linear and attain the semiparametric efficiency bound.

To further highlight the efficiency gains, we compare the semiparametric efficiency bound characterized by \cite{klein2024asymptotic} (see also Section~\ref{sec:conditional_distribution} below) with the explicit asymptotic variances of Chatterjee's rank- and graph-based correlation estimators. The relevant limiting variances were derived by \cite{chhaibi2026martingale} and \cite{gao2026limit}, respectively. We show that these existing estimators generally fail to attain the semiparametric efficiency bound, even in settings where their asymptotic bias is negligible. Thus, in addition to addressing the bias problem, our framework offers a principled route to improving estimation efficiency over existing rank- and graph-based methods.

\paragraph*{Paper organization.} Section~\ref{sec:review} briefly reviews semiparametric efficiency theory and cross-fitted one-step estimation, with the Sobol' index serving as an illustration. Section~\ref{sec:classical} discusses classical dependence functionals formulated as U-statistics, while Section~\ref{sec:conditional_distribution} introduces regression-based measures of functional dependence and derives their EIFs. Section~\ref{sec:estimation} analyzes existing estimators and develops the proposed cross-fitted procedures. Section~\ref{sec:discussion} concludes with further discussion, and Section~\ref{sec:proofs} collects the proofs.

\section{A review of semiparametric efficiency theory}
\label{sec:review}

This section provides a gentle, self-contained review of the
semiparametric efficiency theory used throughout the paper.
For a measurable function $f$, we write
\[
    \P f=\int f\,\d\P
    \qquad\text{and}\qquad
    L_0^2(\P)
    :=
    \left\{
        f:\; \P f=0,\quad \P f^2<\infty
    \right\},
\]
where $L_0^2(\P)$ is equipped with the inner product
$\langle f,g\rangle_\P=\P(fg)$.

\subsection{Basic setup}

Let $Z$ denote a generic observation taking values in a measurable
space $(\mathcal Z,\mathcal A)$, with distribution $\P$ belonging
to a statistical model $\mathcal P$. Consider a smooth
\emph{one-dimensional parametric submodel}
\[
    \{\P_t:t\in(-\delta,\delta)\}
    \subset\mathcal P
\]
passing through $\P=\P_0$.
If the submodel is differentiable in quadratic mean at $t=0$,
it admits a score $g\in L_0^2(\P)$.
When the distributions $\P_t$ admit sufficiently smooth
densities $p_t$ with respect to a common dominating measure,
the score has the familiar representation
\[
    g(Z)
    =
    \left.
    \frac{\d}{\d t}\log p_t(Z)
    \right|_{t=0}.
\]
The collection of all scores arising from regular (i.e., differentiable in quadratic mean at $t=0$) one-dimensional parametric submodels through $\P$ is called the \emph{tangent set} at $\P$,
denoted by
\[
    \dot{\mathcal P}_\P
    \subset L_0^2(\P).
\]
The corresponding \emph{tangent space}, denoted by
$\mathcal T_\P$, is the closure in $L_0^2(\P)$ of the
linear span of $\dot{\mathcal P}_\P$.

Throughout the paper, we consider a scalar parameter
\[
    \psi:\mathcal P\to\R
\]
as the target of inference.

\begin{definition}[Differentiability at $\P$]
\label{def:pathwise_diff}
The parameter $\psi$ is said to be \emph{differentiable at $\P$
relative to the tangent set $\dot{\mathcal P}_\P$},
or simply \emph{differentiable at $\P$}, if there exists
a bounded linear functional
\[
    \dot\psi_\P:\mathcal T_\P\to\R
\]
such that, for every regular one-dimensional submodel
$\{\P_t\}$ through $\P$ with score
$g\in\dot{\mathcal P}_\P$,
\[
    \left.
    \frac{\d}{\d t}\psi(\P_t)
    \right|_{t=0}
    =
    \dot\psi_\P(g).
\]
The map $\dot\psi_\P$ is called the
\emph{pathwise derivative} of $\psi$ at $\P$.
\end{definition}

The Riesz representation theorem provides a convenient
Hilbert-space characterization of the pathwise derivative.

\begin{definition}[IF]
\label{def:if}
Suppose that $\psi:\mathcal P\to\R$ is differentiable
at $\P$. A function $\phi_\P\in L_0^2(\P)$ is
called an \emph{influence function (IF)} of $\psi$ at $\P$ if
\begin{equation}
    \dot\psi_\P(g)
    =
    \P(\phi_\P g)
    =
    \langle\phi_\P,g\rangle_\P,
    \qquad
    \text{for every }g\in\dot{\mathcal P}_\P.
    \label{2.2}
\end{equation}
\end{definition}

By definition, every IF is square-integrable and centered,
so that $\P\phi_\P=0$. However, an IF need not be unique.
Indeed, suppose that $\phi_\P$ satisfies~\eqref{2.2}.
For any $h\in\mathcal T_\P^\perp$, we have
\[
    \P\{(\phi_\P+h)g\}
    =
    \P(\phi_\P g),
    \qquad
    \text{for every }g\in\mathcal T_\P.
\]
Thus, $\phi_\P+h$ represents the same pathwise derivative.
More generally, any two IFs differ by an element of
$\mathcal T_\P^\perp$. This nonuniqueness leads naturally
to the notion of the EIF.

\begin{definition}[EIF]
\label{def:eif}
Suppose that $\psi:\mathcal P\to\R$ is differentiable
at $\P$. The \emph{efficient influence function (EIF)} of $\psi$ at $\P$,
denoted by $\widetilde\psi_\P$, is the unique element
of $\mathcal T_\P$ satisfying
\[
    \dot\psi_\P(g)
    =
    \P(\widetilde\psi_\P g)= \langle\widetilde\psi_\P ,g\rangle_\P,
    \qquad
    \text{for every }g\in\mathcal T_\P.
\]
\end{definition}

Existence and uniqueness of $\widetilde\psi_\P$ follow
directly from the Riesz representation theorem, since
$\dot\psi_\P$ is a bounded linear functional on the
Hilbert space $\mathcal T_\P$.
Equivalently, if $\phi_\P$ is any IF of $\psi$ at $\P$,
then
\[
    \widetilde\psi_\P
    =
    \Pi_{\mathcal T_\P}\phi_\P,
\]
where $\Pi_{\mathcal T_\P}$ denotes the orthogonal
projection in $L_0^2(\P)$ onto $\mathcal T_\P$.
Thus, the EIF is the unique IF belonging to the
tangent space, or equivalently, the orthogonal
projection of any IF onto that space.

A particularly important special case is the unrestricted
nonparametric model. Under standard richness conditions,
its tangent space satisfies
\[
    \mathcal T_\P=L_0^2(\P).
\]
Consequently, $\mathcal T_\P^\perp=\{0\}$, so the IF
representing the pathwise derivative is unique.
In particular, any $\phi_\P\in L_0^2(\P)$ satisfying
\[
    \dot\psi_\P(g)
    =
    \P(\phi_\P g),
    \qquad
    \text{for every }g\in L_0^2(\P),
\]
is necessarily the EIF
\citep[Theorem 2.1]{newey1994asymptotic}.

\subsection{Semiparametric efficiency via the EIF}

We next introduce regular estimators and the corresponding notion of
\emph{semiparametric efficiency}. Throughout the paper, we
consider the usual independent and identically distributed (i.i.d.) sampling scheme, with
$Z_1,\ldots,Z_n$ independently drawn from $\P$.

\begin{definition}[Regular estimator]
\label{def:regular}
An estimator $T_n=T_n(Z_1,\ldots,Z_n)$ is said to be
\emph{regular at $\P$} if there exists a probability distribution
$\mathrm L_\P$, independent of the choice of submodel and of
$t\in\R$, such that, for every regular one-dimensional submodel
$\{\P_t\}$ through $\P$ and every fixed $t\in\R$,
\[
    \sqrt n
    \left\{
        T_n-\psi(\P_{t/\sqrt n})
    \right\}
    \Rightarrow
    \mathrm L_\P
    \qquad
    \text{under }\P_{t/\sqrt n}^{\otimes n},
\]
where the symbol    $\Rightarrow$ denotes weak convergence and
$\P_{t/\sqrt n}^{\otimes n}$ is the $n$-fold product
measure of $\P_{t/\sqrt n}$.
\end{definition}

Thus, regularity requires the first-order limiting distribution
of an estimator, centered at the corresponding parameter value,
to remain unchanged under root-$n$ local perturbations of the
underlying distribution. The EIF introduced in
Definition~\ref{def:eif} determines the minimum asymptotic
variance attainable by regular estimators. This is formalized
by the following convolution theorem.

\begin{theorem}[H\'ajek convolution theorem;
Theorem 18.3 of \cite{kosorok2008introduction}]
\label{thm:convolution}
Suppose that $\psi:\mathcal P\to\R$ is differentiable at
$\P$ with EIF $\widetilde\psi_\P$, and let $T_n$ be
regular at $\P$ relative to $\dot{\mathcal P}_\P$.
Then
\[
    \sqrt n\{T_n-\psi(\P)\}
    \Rightarrow
    \cN\left(0,\P\widetilde\psi_\P^2\right)
    *\mathcal L(M),
\]
where ``$*$'' stands for convolution of distributions and
$\mathcal L(M)$ is the distribution of some random variable $M$.
\end{theorem}

Theorem~\ref{thm:convolution} shows that the limiting
distribution of any regular estimator is the convolution
of a centered Gaussian distribution with variance
$\P\widetilde\psi_\P^2$ and an additional noise distribution.
Consequently, $\P\widetilde\psi_\P^2$ represents the lower bound on the asymptotic variance of any regular estimator of $\psi(P)$ and  is called the \emph{semiparametric efficiency bound} for
estimating $\psi(\P)$. In particular, whenever the limiting
variance exists, it cannot be smaller than this bound.

\begin{definition}[Efficient estimator]
\label{def:efficient}
An estimator $T_n$ is said to be
\emph{efficient at $\P$} if it is regular at $\P$ and further satisfies 
\[
    \sqrt n\{T_n-\psi(\P)\}
    \Rightarrow
    \cN\left(0,\P\widetilde\psi_\P^2\right).
\]
\end{definition}

An important characterization of efficiency is provided by
the asymptotic linear representation of an estimator, which provides a
convenient route to establishing efficiency.

\begin{theorem}[Characterization of efficient estimators;
Lemma 25.23 of \cite{MR1652247}]
\label{thm:efficiency_characterization}
Suppose that $\psi:\mathcal P\to\R$ is differentiable
at $\P$ with EIF $\widetilde\psi_\P$.
Then the following statements are equivalent:
\begin{enumerate}[itemsep=-.5ex,label=(\roman*)]
    \item $T_n$ is efficient at $\P$ relative to
    $\dot{\mathcal P}_\P$;

    \item $T_n$ admits the asymptotic linear representation
    \[
        T_n-\psi(\P)
        =
        \frac1n\sum_{i=1}^n
        \widetilde\psi_\P(Z_i)
        +o_\P(n^{-1/2}).
    \]
\end{enumerate}
In particular, either condition implies that $T_n$ is regular at $\P$ relative to $\dot{\mathcal P}_\P$.
\end{theorem}

\subsection{A general strategy toward efficiency}\label{sec:general-method}

We now introduce the classical one-step construction
\citep[Chapter 7]{bickel1993efficient} for obtaining efficient estimators.
In the following, consider $\QQ_n$ to be an initial estimator of $\P$, and define the corresponding
one-step estimator of $\psi(\P)$ by
\begin{align}
    \widehat\psi_n^*
    :=
    \psi(\QQ_n)
    +
    \PP_n\widetilde\psi_{\QQ_n},
    \label{eq:one-step}
\end{align}
where $\PP_n$ denotes the empirical measure of
$Z_1,\ldots,Z_n$.

To understand the role of the one-step correction, define
\[
    \Rem(\QQ_n,\P)
    :=
    \psi(\QQ_n)-\psi(\P)
    +
    \P\widetilde\psi_{\QQ_n}.
\]
Then~\eqref{eq:one-step} admits the exact decomposition
\begin{align}
    \widehat\psi_n^*
    -
    \psi(\P)
    &=
    (\PP_n-\P)\widetilde\psi_\P
+
    (\PP_n-\P)
    \left(
        \widetilde\psi_{\QQ_n}
        -
        \widetilde\psi_\P
    \right)
    +
    \Rem(\QQ_n,\P).
    \label{eq:one-step-decomp}
\end{align}
The first term in~\eqref{eq:one-step-decomp} is the desired first-order
term. The second is an \emph{empirical-process} remainder, while the third
measures the bias remaining after the one-step correction.

Now suppose that
\begin{align}
    \left\|
        \widetilde\psi_{\QQ_n}
        -
        \widetilde\psi_\P
    \right\|_{L^2(\P)}
    =
    o_\P(1),
    \qquad
    \Rem(\QQ_n,\P)
    =
    o_\P(n^{-1/2}).
\end{align}
If, in addition, $\widetilde\psi_{\QQ_n}$ belongs with probability
tending to one to a suitable $\P$-Donsker class containing
$\widetilde\psi_\P$, then asymptotic equicontinuity yields
\[
    (\PP_n-\P)
    \left(
        \widetilde\psi_{\QQ_n}
        -
        \widetilde\psi_\P
    \right)
    =
    o_\P(n^{-1/2}).
\]
It follows that
\begin{align}
    \widehat\psi_n^*
    -
    \psi(\P)
    =
    \frac{1}{n}
    \sum_{i=1}^n
    \widetilde\psi_\P(Z_i)
    +
    o_\P(n^{-1/2}).
\end{align}
By Theorem~\ref{thm:efficiency_characterization},
$\widehat\psi_n^*$ is then regular and semiparametrically efficient at
$\P$.

In modern applications, however, the initial estimator $\QQ_n$ often
involves flexible nonparametric or machine-learning methods, for which
verifying Donsker-type complexity conditions can be restrictive. Such
conditions can be avoided through sample-splitting and cross-fitting
\citep{klaassen1987consistent,newey2018cross,chernozhukov2018double}. Let
$\{\mathcal I_1,\ldots,\mathcal I_K\}$ with $K\geq 2$ be a deterministic partition of
$\{1,\ldots,n\}$, and write
\[
    n_k:=|\mathcal I_k|.
\]
For each $k$, let $\QQ_k$ be constructed using only observations with
indices outside $\mathcal I_k$, and define the empirical measure on the
$k$th validation fold by
\[
    \PP_{n,k}
    :=
    \frac{1}{n_k}
    \sum_{i\in\mathcal I_k}
    \delta_{Z_i}
\]
with $\delta_z$ representing the Dirac measure at $z$. The cross-fitted one-step estimator is defined by
\begin{align}
    \widehat\psi_n^{**}
    &:=
    \sum_{k=1}^K
    \frac{n_k}{n}
    \left\{
        \psi(\QQ_k)
        +
        \PP_{n,k}\widetilde\psi_{\QQ_k}
    \right\}
=
    \frac{1}{n}
    \sum_{k=1}^K
    \sum_{i\in\mathcal I_k}
    \left\{
        \psi(\QQ_k)
        +
        \widetilde\psi_{\QQ_k}(Z_i)
    \right\}.
\end{align}

Cross-fitting removes the need for Donsker conditions because
$\widetilde\psi_{\QQ_k}$ is constructed independently of the observations
in the validation fold $\mathcal I_k$. The following result gives a
convenient sufficient condition for semiparametric efficiency.

\begin{theorem}[Efficiency of cross-fitted one-step estimators]
\label{thm:cf_onestep_efficiency}
Suppose that $\psi:\mathcal P\to\R$ is differentiable at $\P$ with EIF
$\widetilde\psi_\P\in L_0^2(\P)$. Assume that the number of folds $K$ is
fixed and that, for some constant $c>0$,
\[
    \min_{1\leq k\leq K}
    \frac{n_k}{n}
    \geq c
\]
for all sufficiently large $n$. Suppose further that, uniformly over
$k=1,\ldots,K$,
\[
    \left\|
        \widetilde\psi_{\QQ_k}
        -
        \widetilde\psi_\P
    \right\|_{L^2(\P)}
    =
    o_\P(1)~~~{\rm and}~~~
    \Rem(\QQ_k,\P)
    =
    o_\P(n^{-1/2}).
\]
Then
\begin{align}
    \widehat\psi_n^{**}
    -
    \psi(\P)
    =
    \frac{1}{n}
    \sum_{i=1}^n
    \widetilde\psi_\P(Z_i)
    +
    o_\P(n^{-1/2}),
    \label{eq:cf-one-step-al}
\end{align}
and consequently $\widehat\psi_n^{**}$ is regular and semiparametrically efficient at
$\P$.
\end{theorem}

\begin{remark}\label{remark:general-dml}
In applications, we may adopt a more flexible version of the
preceding construction, allowing the estimated EIF
$\widetilde\psi_{\QQ_k}$ to depend on certain quantities computed
from the validation sample $\PP_{n,k}$. As illustrated below,
the key requirement is typically that the preliminary
nonparametric estimators of the infinite-dimensional nuisance
parameters be constructed independently of the validation
sample. Other components, such as empirical marginal
distributions or moments, may be estimated using the validation
sample itself, provided that their additional stochastic
contributions are properly accounted for in the asymptotic
analysis. Thus, the independence requirement underlying
cross-fitting need not extend to every component of the
estimated EIF.  Theorem \ref{thm:cf_onestep_efficiency} does not directly apply to such constructions, and their asymptotic linearity must be verified separately, as illustrated in Theorem \ref{thm:sobol-efficiency} below.
\end{remark}

\subsection{Example: EIF-based estimation of the Sobol' index}
\label{sec:sobol-efficient-estimation}

We illustrate the general strategy of Section~\ref{sec:general-method}
using the Sobol' index
\citep{doksum1995nonparametric,da2013efficient,klein2024asymptotic,da2026sensitivity}.
Let $Z=(X,Y)\sim\P$, where $X\in\cX\subset\R^d$ and $Y\in\R$, and write
\[
    m(x):=\E[Y\mid X=x],\qquad
    \mu:=\E[Y],\qquad
    v:=\Var(Y)>0.
\]
The Sobol' index measures the proportion of the variance of $Y$
explained by $X$ and is defined by
\[
    \rho_{\mathrm{Sob}}(\P)
    :=\frac{\Var\{\E[Y\mid X]\}}{\Var(Y)}
    =\frac{\E[m(X)^2]-\mu^2}{v}.
\]

We first derive a candidate EIF, following the argument in
\citet[Section~2]{doksum1995nonparametric}.
To this end, first fix a bounded score $s\in L_0^2(\P)$ and consider the submodel
\[
    \d\P_\epsilon(z)=\{1+\epsilon s(z)\}\,\d\P(z),
\]
where $|\epsilon|$ is sufficiently small that
$1+\epsilon s(z)>0$ almost surely. Let
$m_\epsilon(x):=\E_\epsilon[Y\mid X=x]$ denote the conditional expectation of $Y$ given $X=x$ with respect to the distribution $\P_\epsilon$.
We begin by deriving the pathwise derivative of $m(x)$.

For any bounded measurable function $h:\cX\to\R$,
the change-of-measure identity and the tower property give
\begin{align*}
    \E_\epsilon[Yh(X)]
    &=\E[Yh(X)\{1+\epsilon s(Z)\}]=\E\!\left[h(X)\E[Y\{1+\epsilon s(Z)\}\mid X]\right].
\end{align*}
On the other hand, the defining property of the conditional mean yields
\begin{align*}
    \E_\epsilon[Yh(X)]
    &=\E_\epsilon[m_\epsilon(X)h(X)]=\E\!\left[
        h(X)m_\epsilon(X)\E[1+\epsilon s(Z)\mid X]
      \right].
\end{align*}
Since these equalities hold for every bounded measurable $h$,
we obtain, for $\P_X$-almost every $x$,
\[
    m_\epsilon(x)
    =\frac{\E[Y\{1+\epsilon s(Z)\}\mid X=x]}
           {\E[1+\epsilon s(Z)\mid X=x]}
    =\frac{m(x)+\epsilon\E[Ys(Z)\mid X=x]}
           {1+\epsilon\E[s(Z)\mid X=x]}.
\]
Subtracting $m(x)$ and rearranging gives the exact identity
\[
    m_\epsilon(x)-m(x)
    =\epsilon\,
      \frac{\E[\{Y-m(X)\}s(Z)\mid X=x]}
           {1+\epsilon\E[s(Z)\mid X=x]}.
\]
The denominator converges to one and is bounded away from zero
for sufficiently small $|\epsilon|$. Consequently,
\[
    \dot m(x)
    :=\left.\frac{\d}{\d\epsilon}m_\epsilon(x)\right|_{\epsilon=0}
    =\E[\{Y-m(X)\}s(Z)\mid X=x].
\]
We next differentiate the quadratic functional $\E[m(X)^2]$.
Because $m_\epsilon(X)^2$ depends only on $X$,
\begin{align*}
    \E_\epsilon[m_\epsilon(X)^2]
    &=\E[m_\epsilon(X)^2\{1+\epsilon s(Z)\}]=\E\!\left[
        m_\epsilon(X)^2\{1+\epsilon\E[s(Z)\mid X]\}
      \right].
\end{align*}
Inserting the preceding identity for $m_\epsilon-m$ and
expanding the square, we obtain
\begin{align*}
    \E_\epsilon[m_\epsilon(X)^2]
    ={}&\E[m(X)^2]+\epsilon\E\!\left[
        m(X)^2\E[s(Z)\mid X]
        +2m(X)\E[\{Y-m(X)\}s(Z)\mid X]
      \right]\\
    &+\epsilon^2\E\!\left[
       \frac{\E[\{Y-m(X)\}s(Z)\mid X]^2}
            {1+\epsilon\E[s(Z)\mid X]}
      \right]\\
    ={}&\E[m(X)^2]
       +\epsilon\E\!\left[
          \{m(X)^2+2m(X)(Y-m(X))\}s(Z)
        \right]
       +O(\epsilon^2).
\end{align*}
To justify the remainder, note that conditional Jensen's inequality
and the boundedness of $s$ give
\[
    \E\!\left[\E[\{Y-m(X)\}s(Z)\mid X]^2\right]
    \le \|s\|_\infty^2\E[\{Y-m(X)\}^2]<\infty,
\]
while the denominator is at least $1-|\epsilon|\|s\|_\infty$.
It follows that
\[
    \left.\frac{\d}{\d\epsilon}
    \E_\epsilon[m_\epsilon(X)^2]\right|_{\epsilon=0}
    =\E\!\left[
        \{2m(X)Y-m(X)^2\}s(Z)
      \right].
\]
It remains to differentiate the marginal mean and variance.
Writing $\mu_\epsilon:=\E_\epsilon[Y]$ and
$v_\epsilon:=\Var_\epsilon(Y)$, we have
\[
    \mu_\epsilon=\E[Y\{1+\epsilon s(Z)\}],
    \qquad
    v_\epsilon=\E_\epsilon[Y^2]-\mu_\epsilon^2.
\]
Since $\E[s(Z)]=0$, their derivatives are
\begin{align*}
    \dot\mu
    &:=\left.\frac{\d}{\d\epsilon}\mu_\epsilon\right|_{\epsilon=0}
      =\E[(Y-\mu)s(Z)],~~~
    \dot v
    &:=\left.\frac{\d}{\d\epsilon}v_\epsilon\right|_{\epsilon=0}
      =\E[Y^2s(Z)]-2\mu\dot\mu
      =\E[\{(Y-\mu)^2-v\}s(Z)].
\end{align*}
Combining these identities, the derivative of the numerator is
\begin{align*}
    \left.\frac{\d}{\d\epsilon}
    \{\E_\epsilon[m_\epsilon(X)^2]-\mu_\epsilon^2\}
    \right|_{\epsilon=0}&=\E\!\left[
       \{2m(X)Y-m(X)^2-2\mu(Y-\mu)\}s(Z)
      \right]\\
    &=\E\!\left[
       \big\{2\{m(X)-\mu\}\{Y-m(X)\}
              +\{m(X)-\mu\}^2\big\}s(Z)
      \right],
\end{align*}
where the last equality uses $\E[s(Z)]=0$.
Applying the quotient rule, and using
$\E[m(X)^2]-\mu^2=v\rho_{\mathrm{Sob}}(\P)$, now gives
\begin{align*}
    \left.\frac{\d}{\d\epsilon}
       \rho_{\mathrm{Sob}}(\P_\epsilon)
    \right|_{\epsilon=0}
    =\frac1v\E\!\Big[
       \big\{&2\{m(X)-\mu\}\{Y-m(X)\}
        +\{m(X)-\mu\}^2-\rho_{\mathrm{Sob}}(\P)(Y-\mu)^2\big\}s(Z)
      \Big].
\end{align*}
Thus, the candidate EIF is given by 
\[
    \begin{aligned}
    \widetilde\rho_{\mathrm{Sob},\P}(x,y)
    =\frac1v\Big[
       &2\{m(x)-\mu\}\{y-m(x)\}
        +\{m(x)-\mu\}^2-\rho_{\mathrm{Sob}}(\P)(y-\mu)^2
      \Big].
    \end{aligned}
\]
This function is centered: the first term has mean zero by
$\E[Y-m(X)\mid X]=0$, and the expectations of the remaining
two terms cancel because
$\E[\{m(X)-\mu\}^2]=v\rho_{\mathrm{Sob}}(\P)$. 
If $\E[Y^4]<\infty$, conditional Jensen's inequality ensures that the candidate belongs to $L_0^2(\P)$. Along parametric submodels for which the relevant moment functionals are differentiable and differentiation under the integral sign is justified, this extension represents the pathwise derivative. Consequently, the candidate is the EIF relative to such a class of regular submodels, provided that their scores span a dense subspace of $L_0^2(\P)$; see also \citet[Section~2]{doksum1995nonparametric}.

We now turn to estimation. Following the cross-fitting scheme in
Section~\ref{sec:general-method}, partition the observations into
$K$ validation folds $\mathcal I_1,\ldots,\mathcal I_K$,
with $n_k=|\mathcal I_k|$. For each $k$, construct an estimator
$\widehat m_k$ of $m$ using only observations outside
$\mathcal I_k$, and define
\[
    \widehat\mu_k:=\PP_{n,k}Y,\qquad
    \widehat\mu:=\PP_nY,\qquad
    \widehat v:=\PP_n(Y-\widehat\mu)^2.
\]
We use a pooled estimate of the marginal variance, while
retaining fold-specific estimates of the marginal mean in the
numerator. The resulting plug-in estimator is
\[
    \widehat\rho_{\mathrm{Sob}}^{\mathrm{plug}}
    :=\frac1{\widehat v}
       \sum_{k=1}^K\frac{n_k}{n}
       \left\{
          \PP_{n,k}\widehat m_k(X)^2-\widehat\mu_k^2
       \right\}.
\]
Motivated by the EIF derived above, define its empirical analogue
on fold $k$ by
\begin{align*}
    \widehat{\widetilde\rho}_{\mathrm{Sob},k}(x,y)
    :=\frac1{\widehat v}\Big[
       &\widehat m_k(x)^2
        -\PP_{n,k}\widehat m_k(X)^2
        +2\widehat m_k(x)\{y-\widehat m_k(x)\}-2\widehat\mu_k(y-\widehat\mu_k)
        -\widehat\rho_{\mathrm{Sob}}^{\mathrm{plug}}
          \{(y-\widehat\mu)^2-\widehat v\}
      \Big].
\end{align*}
Only the infinite-dimensional nuisance function $m$ is
estimated on an independent training sample; the empirical
moments in this expression may depend on the validation
sample. As the following algebra shows, this dependence
simplifies in the final estimator.

The cross-fitted one-step estimator is
\begin{align}
    \widehat\rho_{\mathrm{Sob},n}^{\mathrm{eff}}
    &:=\widehat\rho_{\mathrm{Sob}}^{\mathrm{plug}}
       +\sum_{k=1}^K\frac{n_k}{n}
         \PP_{n,k}\widehat{\widetilde\rho}_{\mathrm{Sob},k}
         =\frac1{\widehat v}
       \sum_{k=1}^K\frac{n_k}{n}
       \left[
         \PP_{n,k}\{2\widehat m_k(X)Y-\widehat m_k(X)^2\}
         -\widehat\mu_k^2
       \right],
    \label{eq:sobol-eff-final}
\end{align}
where the second equality uses the facts that
\begin{align*}
    \PP_{n,k}\{\widehat m_k(X)^2
        -\PP_{n,k}\widehat m_k(X)^2\}=0,~~
    \PP_{n,k}(Y-\widehat\mu_k)=0,~~
    \sum_{k=1}^K\frac{n_k}{n}\PP_{n,k}
       \{(Y-\widehat\mu)^2-\widehat v\}=0.
\end{align*}
The one-step correction yields a Neyman-orthogonal moment
condition with respect to the nuisance regression function
$m$ \citep[Section 2.2]{chernozhukov2018double}. Indeed, for any
square-integrable function $q$,
\begin{align*}
&\E[2q(X)Y-q(X)^2]-\E[m(X)^2]=
2\E[q(X)m(X)]-\E[q(X)^2]-\E[m(X)^2]=
-\|q-m\|_{L^2(\P_X)}^2.
\end{align*}
Thus, the population moment is insensitive to first-order
perturbations of $m$, and its bias is exactly quadratic
in the nuisance estimation error. This property underlies
the relatively mild convergence-rate requirement imposed
on the preliminary regression estimator. See also
\citet{doksum1995nonparametric} for an earlier development
of this one-step correction.

The following result gives sufficient conditions for the
cross-fitted estimator to achieve the EIF-based asymptotic
linear representation. In particular, sample-splitting permits
us to avoid Donsker conditions on the first-stage estimator.

\begin{theorem}
\label{thm:sobol-efficiency}
Suppose that $Z_1,\ldots,Z_n$ are i.i.d.\ from $\P$,
with $\E[Y^4]<\infty$ and $\Var(Y)>0$.
Let $\{\mathcal I_1,\ldots,\mathcal I_K\}$ be a partition
into a fixed number of folds satisfying $n_k\asymp n$.
For each $k$, suppose that $\widehat m_k$ is constructed
using only observations outside $\mathcal I_k$ and satisfies
\[
    \max_{1\le k\le K}
       \|\widehat m_k-m\|_{L^2(\P_X)}
       =o_\P(n^{-1/4}),
    \qquad
    \max_{1\le k\le K}
       \|\widehat m_k-m\|_{L^4(\P_X)}
       =o_\P(1).
\]
Then the estimator $\widehat\rho_{\mathrm{Sob},n}^{\mathrm{eff}}$
in~\eqref{eq:sobol-eff-final} satisfies
\[
    \widehat\rho_{\mathrm{Sob},n}^{\mathrm{eff}}
       -\rho_{\mathrm{Sob}}(\P)
    =\frac1n\sum_{i=1}^n
       \widetilde\rho_{\mathrm{Sob},\P}(X_i,Y_i)
       +o_\P(n^{-1/2}).
\]
\end{theorem}

\section{Classical measures of dependence}
\label{sec:classical}

Throughout the remainder of this paper, let
\[
    Z=(X,Y)\sim \P\in\cP,
\]
where $X\in\cX$ and $Y\in\cY$ are two random variables with marginal measures $\P_X$ and $\P_Y$, respectively. Broadly speaking,
a measure of dependence is a scalar-valued functional
\[
    \Xi:\mathcal P\to\R
\]
designed to quantify the strength of dependence between $X$ and $Y$.

We begin with a broad class of classical dependence
measures whose population versions can be represented as expectations of
\emph{symmetric kernels}; see, e.g., \cite{MR3737306}, \cite{MR4185806}, and \cite{shi2020power}. This class contains many classical rank-based and
kernel-based measures of dependence and provides a particularly simple
setting in which the EIF and the efficiency of their usual sample
analogues can be identified explicitly.

Specifically, let $h:\mathcal Z^m\to\R$ be a measurable symmetric kernel of fixed order
$m\geq1$. Consider the functional
\begin{align}\label{eq:u-md}
    \psi_h(\P)
    :=
    \E_\P\{h(Z_1,\ldots,Z_m)\}
=
    \int h(z_1,\ldots,z_m)
    \prod_{j=1}^m \d\P(z_j),
\end{align}
where $Z_1,\ldots,Z_m$ are independent copies of $Z$. Its natural sample
analogue is the U-statistic
\[
    U_n(h)
    :=
    {n\choose m}^{-1}
    \sum_{1\leq i_1<\cdots<i_m\leq n}
    h(Z_{i_1},\ldots,Z_{i_m}).
\]
Throughout this section, we work under the unrestricted nonparametric
model, so that
\[
    \mathcal T_\P=L_0^2(\P).
\]
For later use, if $a:\mathcal Z^m\to\R$ is not necessarily symmetric,
write
\[
    \operatorname{Sym}_m a(z_1,\ldots,z_m)
    :=
    \frac{1}{m!}
    \sum_{\pi\in\mathfrak S_m}
    a(z_{\pi_1},\ldots,z_{\pi_m})
\]
for its symmetrization.

\begin{example}[Kendall's tau]
\label{example:kendall}
Suppose that $Z=(X,Y)$, where $X$ and $Y$ are real valued. Kendall's
tau admits the representation
\[
    \tau(\P)
    =
    \E_\P [ h_\tau(Z_1,Z_2)]  =
    \E_\P\left[
        \sgn(X_1-X_2)\sgn(Y_1-Y_2)
    \right]
\]
with
$
    h_\tau(z_1,z_2)
    :=
    \sgn(x_1-x_2)\sgn(y_1-y_2)$ and $\sgn(x):=\ind(x> 0)-\ind(x<0)$.
Its usual sample version is exactly the second-order U-statistic
$
    \widehat\tau_n=U_n(h_\tau).
$
\end{example}

\begin{example}[Spearman's rho]
\label{example:spearman}
Under the same setting, and assuming continuous marginals, Spearman's rho
can be written as
\[
    \rho(\P)
    =
    12\,\E_\P\{F_X(X)F_Y(Y)\}-3=
    3\,\E_\P\left[
        \sgn(X_1-X_2)\sgn(Y_1-Y_3)
    \right].
\]
Define the symmetric third-order kernel
\begin{align}
    h_\rho(z_1,z_2,z_3)
    :=
    \frac{1}{2}
    \sum_{\pi\in\mathfrak S_3}
    \sgn(x_{\pi_1}-x_{\pi_2})
    \sgn(y_{\pi_1}-y_{\pi_3}).
\end{align}
Then
$    \rho(\P)
    =
    \E_\P\{h_\rho(Z_1,Z_2,Z_3)\}.$ 
Let $U_{n,\rho}:=U_n(h_\rho)$. The usual sample Spearman correlation
$\widehat\rho_n$ satisfies the exact identity
\begin{align} \label{spearequiv}
    \widehat\rho_n
    =
    \frac{n-2}{n+1}U_{n,\rho}
    +
    \frac{3}{n+1}\widehat\tau_n.
\end{align}
In particular,
$\widehat\rho_n-U_{n,\rho}=O(n^{-1})$.
\end{example}

\begin{example}[Hoeffding's $D$ and the Blum--Kiefer--Rosenblatt
coefficient]
\label{example:hoeffding_bkr}
Let
\[
    \Delta_\P(x,y)
    :=
    F_{XY}(x,y)-F_X(x)F_Y(y).
\]
For continuous bivariate distributions, Hoeffding's dependence measure \citep{MR0029139} and the Blum--Kiefer--Rosenblatt dependence measure \citep{MR0125690} are respectively
\[
    D(\P)
    :=
    \int_{\R^2}
    \Delta_\P^2(x,y)
    \d\P_{XY}(x,y)
~~~{\rm and}~~~    R(\P)
    :=
    \int_{\R^2}
    \Delta_\P^2(x,y)
    \d\P_X(x)\d\P_Y(y).
\]
Both admit finite-order U-statistic representations. To make this explicit, define
\begin{align*}
    a_D(z_1,\ldots,z_5)
    &:=
    \frac14
    \{\ind(x_1\leq x_5)-\ind(x_2\leq x_5)\}
    \{\ind(x_3\leq x_5)-\ind(x_4\leq x_5)\}
    \\
    &\qquad\times
    \{\ind(y_1\leq y_5)-\ind(y_2\leq y_5)\}
    \{\ind(y_3\leq y_5)-\ind(y_4\leq y_5)\}
\end{align*}
and $
    h_D:=\operatorname{Sym}_5 a_D.$
Then
\[
    D(\P)
    =
    \E_\P\{h_D(Z_1,\ldots,Z_5)\}.
\]
Similarly, define
\begin{align*}
    a_R(z_1,\ldots,z_6)
    &:=
    \frac14
    \{\ind(x_1\leq x_5)-\ind(x_2\leq x_5)\}
    \{\ind(x_3\leq x_5)-\ind(x_4\leq x_5)\}
    \\
    &\qquad\times
    \{\ind(y_1\leq y_6)-\ind(y_2\leq y_6)\}
    \{\ind(y_3\leq y_6)-\ind(y_4\leq y_6)\},
\end{align*}
and let $h_R:=\operatorname{Sym}_6 a_R$.
Then
\[
    R(\P)
    =
    \E_\P\{h_R(Z_1,\ldots,Z_6)\}.
\]
Their usual unbiased sample versions are therefore U-statistics of orders
five and six, respectively.
\end{example}

\begin{example}[Bergsma--Dassios--Yanagimoto's $\tau^*$]
\label{example:tau_star}
For $x_1,\ldots,x_4\in\R$, define
\[
    a(x_1,x_2,x_3,x_4)
    :=
    \sgn\left(
        |x_1-x_2|
        +
        |x_3-x_4|
        -
        |x_1-x_3|
        -
        |x_2-x_4|
    \right).
\]
Then Bergsma--Dassios--Yanagimoto's sign covariance \citep{yanagimoto1970measures,MR3178526} can be represented as
\[
    \tau^*(\P)
    :=
    \E_\P\left[
        a(X_1,X_2,X_3,X_4)
        a(Y_1,Y_2,Y_3,Y_4)
    \right].
\]
Thus, with
$
    h_{\tau^*}
    :=
    \operatorname{Sym}_4
    \left[
        a(x_1,x_2,x_3,x_4)
        a(y_1,y_2,y_3,y_4)
    \right]$, we have
\[
    \tau^*(\P)
    =
    \E_\P\{h_{\tau^*}(Z_1,\ldots,Z_4)\}.
\]
The usual unbiased sample version is therefore a fourth-order
rank-based U-statistic.
\end{example}

\begin{example}[Distance covariance]
\label{example:dcov}
Let $X\in\R^p$ and $Y\in\R^q$, and write
\[
    a_{ij}:=\|X_i-X_j\|,
    \qquad
    b_{ij}:=\|Y_i-Y_j\|.
\]
The squared distance covariance \citep{MR2382665} admits the population representation
\[
     \operatorname{dCov}^2(\P)
    =
    \E_\P(a_{12}b_{12})
    +
    \E_\P(a_{12})\E_\P(b_{12})
    -
    2\E_\P(a_{12}b_{13}).
\]
Equivalently, define
$
    h_{\mathrm{dCov}}
    :=
    \operatorname{Sym}_4
    \left(
        a_{12}b_{12}
        +
        a_{12}b_{34}
        -
        2a_{12}b_{13}
    \right).
$
Then
\[
    \operatorname{dCov}^2(\P)
    =
    \E_\P\{
        h_{\mathrm{dCov}}(Z_1,Z_2,Z_3,Z_4)
    \}.
\]
The unbiased estimator of squared distance covariance is the corresponding
fourth-order U-statistic.
\end{example}

\begin{example}[Hilbert--Schmidt independence criterion]
\label{example:hsic}
Let $k:\mathcal X^2\to\R$ and
$\ell:\mathcal Y^2\to\R$ be symmetric positive-definite kernels, and
write
\[
    k_{ij}:=k(X_i,X_j),
    \qquad
    \ell_{ij}:=\ell(Y_i,Y_j).
\]
The population Hilbert--Schmidt independence criterion \citep{gretton2005measuring} can be written as
\begin{align}
    \operatorname{HSIC}(\P)
    &=
    \E_\P(k_{12}\ell_{12})
    +
    \E_\P(k_{12})\E_\P(\ell_{12})
    -
    2\E_\P(k_{12}\ell_{13}).
\end{align}
Define $
    h_{\mathrm{HSIC}}
    :=
    \operatorname{Sym}_4
    \left(
        k_{12}\ell_{12}
        +
        k_{12}\ell_{34}
        -
        2k_{12}\ell_{13}
    \right).
$
Then
\[
    \operatorname{HSIC}(\P)
    =
    \E_\P\{
        h_{\mathrm{HSIC}}(Z_1,Z_2,Z_3,Z_4)
    \}.
\]
For bounded kernels, $h_{\mathrm{HSIC}}$ is automatically square
integrable, and the usual unbiased HSIC estimator is a fourth-order
U-statistic.
\end{example}

The preceding examples all fall under the general representation \eqref{eq:u-md}. In the following discussion we  characterize the EIF of such functionals. Define the first-order Hoeffding projection of $h$ \citep{MR0026294} by
\begin{align}
    h_{1,\P}(z)
    :=
    \E_\P\{h(z,Z_2,\ldots,Z_m)\}
    -
    \psi_h(\P).
\end{align}

\begin{proposition}[EIF for U-statistic functionals]\label{prop:u_eif}
Suppose that $h$ is a measurable symmetric kernel satisfying $\E_\P[h^2(Z_1,\ldots,Z_m)]<\infty$, and $\psi_h$ is pathwise differentiable at $\P$. Then the EIF of $\psi_h$ is given by
\[
    \widetilde\psi_{h,\P}(z)=m h_{1,\P}(z).
\]
\end{proposition}

The same first-order projection governs the asymptotic behavior of the
sample U-statistic. Thus, in this setting, the classical Hoeffding
decomposition and the semiparametric efficiency expansion coincide
exactly at first order.

\begin{proposition}[Efficiency of the sample U-statistic]
\label{prop:u_efficiency}
Under the conditions of Proposition~\ref{prop:u_eif},
\begin{align}
    U_n(h)-\psi_h(\P)
    =
    \frac{1}{n}
    \sum_{i=1}^n
    \widetilde\psi_{h,\P}(Z_i)
    +
    o_\P(n^{-1/2}).
\end{align}
Consequently, $U_n(h)$ is efficient for
estimating $\psi_h(\P)$ in the unrestricted nonparametric model.
\end{proposition}

The preceding result immediately yields the efficiency of the classical sample versions of the dependence measures above.

\begin{corollary}
\label{cor:classical_efficiency}
\begin{enumerate}[itemsep=-.5ex,label=(\roman*)]
\item In the unrestricted nonparametric model, the usual unbiased U-statistic estimators of Kendall's tau,  Hoeffding's $D$, Blum--Kiefer--Rosenblatt's $R$, Bergsma--Dassios--Yanagimoto's $\tau^*$ are semiparametrically efficient for their respective population functionals. The same conclusion holds for the usual sample Spearman correlation, by the asymptotic equivalence presented in \eqref{spearequiv}.

\item The same conclusion holds for squared distance covariance and HSIC, provided that their corresponding U-statistic kernels are square-integrable and their population functionals are pathwise differentiable at $\P$. In particular, the latter conditions hold for HSIC with bounded kernels.
\end{enumerate}
\end{corollary}

\begin{remark}[Degeneracy under independence]
\label{rem:classical_degeneracy}
For many measures of independence, including Hoeffding's $D$,
Blum--Kiefer--Rosenblatt's $R$, Bergsma--Dassios--Yanagimoto's
$\tau^*$, squared distance covariance, and HSIC, the first-order
Hoeffding projection vanishes under independence. Consequently,
$
    \widetilde\psi_{h,\P}=0
$
at an independence distribution, and the semiparametric efficiency bound
is zero. Proposition~\ref{prop:u_efficiency} remains valid, but the
root-$n$ limiting distribution is then degenerate and a nondegenerate
null limit arises only at a faster, higher-order scale; in the literature of U-statistics, this leads to the celebrated degenerate U-statistics theory \citep{MR0026294,nolan1987u}. We will revisit this degeneracy phenomenon
in the later sections.
\end{remark}

\section{Regression-based measures of dependence}
\label{sec:conditional_distribution}

A line of research initiated by
\cite{trutschnig2011strong} and \cite{MR3024030},
and subsequently popularized by
\cite{chatterjee2020new}, focuses on constructing
directed measures of dependence $\xi(\P)$ satisfying
the following two properties:
\begin{enumerate}[itemsep=-.5ex,label=(\roman*)]
    \item $\xi(\P)=0$ if and only if $X$ and $Y$ are
    independent;
    \item $\xi(\P)=1$ if and only if $Y$ is almost
    surely a measurable function of $X$.
\end{enumerate}
To introduce these measures, suppose that
$X\in\cX\subseteq\R^d$, for some fixed $d\geq1$,
and $Y\in\R$. For $t\in\R$ and $x\in\cX$, let
\[
    Q_t(x)
    :=
    \P(Y\leq t\mid X=x)
    =
    \E_\P[\ind(Y\leq t)\mid X=x],
    \qquad
    F_Y(t):=\P(Y\leq t).
\]
Throughout this section, we use $t$ for a generic
integration threshold and write $Q_y(x)$ when the
threshold is evaluated at a particular value $y$.

The population version of Chatterjee's correlation
coefficient, originally introduced by
\cite{MR3024030} and henceforth denoted by
$\xi(\P)$, is defined as
\begin{equation}
    \xi(\P)
    :=
    \frac{
        \displaystyle
        \int_{\R}
        \Var_\P\{Q_t(X)\}\,\d\P_Y(t)
    }{
        \displaystyle
        \int_{\R}
        \Var_\P\{\ind(Y\leq t)\}\,\d\P_Y(t)
    }.
    \label{eq:chatterjee_population}
\end{equation}
When $Y$ is continuously distributed, the denominator
simplifies to
\begin{align*}
    \int_{\R}
    \Var_\P\{\ind(Y\leq t)\}\,\d\P_Y(t)
    &=
    \int_{\R}
    F_Y(t)\{1-F_Y(t)\}\,\d\P_Y(t)=
    \int_0^1u(1-u)\,\d u
    =\frac16,
\end{align*}
and consequently,
\[
    \xi(\P)
    =
    6\int_{\R}
    \Var_\P\{Q_t(X)\}\,\d\P_Y(t).
\]
In particular, $\xi(\P)$ vanishes precisely under
independence and equals one precisely when $Y$ is
almost surely a measurable function of $X$.

Building on the same principle, recent work by
\cite{ansari2026quantifying} and
\cite{ansari2025ordering} has introduced broader
classes of regression-based measures of dependence,
defined as follows.

\begin{definition}[Regression-based measures of dependence]
Let $\phi:[-1,1]\to\R$ be convex, strictly
convex at zero, and satisfy $\phi(0)=0$.
We consider the following two classes
\begin{equation}
\begin{aligned}
    \xi_\phi(\P)
    &:=
    \frac1{\alpha_\phi}
    \int_{\R}\int_{\cX}
    \phi\{Q_t(x)-F_Y(t)\}
    \,\d\P_X(x)\d\P_Y(t),\\
    \Lambda_\phi(\P)
    &:=
    \frac1{\beta_\phi}
    \int_{\R}\int_{\cX^2}
    \phi\{Q_t(x_1)-Q_t(x_2)\}
    \,\d(\P_X\otimes\P_X)(x_1,x_2)
    \d\P_Y(t),
\end{aligned}
\label{eq:rmd}
\end{equation}
where the normalizing constants are given by
\begin{align}
    \alpha_\phi
    &:=
    \int_{\R^2}
    \phi\{\ind(y_1\leq t)-F_Y(t)\}
    \,\d\P_Y(y_1)\d\P_Y(t),
    \notag\\
    \beta_\phi
    &:=
    \int_{\R^3}
    \phi\{\ind(y_1\leq t)-\ind(y_2\leq t)\}
    \,\d\P_Y(y_1)\d\P_Y(y_2)\d\P_Y(t).
    \label{3.20}
\end{align}
\end{definition}

Under the stated conditions on $\phi$, both
$\xi_\phi(\P)$ and $\Lambda_\phi(\P)$ take values
in the interval $[0,1]$. Moreover, each vanishes if and only if
$X$ and $Y$ are independent, and equals one if and
only if $Y$ is almost surely a measurable function
of $X$ \citep[Theorem 4.1]{ansari2025ordering}.

For the semiparametric analysis below, we restrict
attention to continuously distributed $Y$. In this
case, the normalizing constants in~\eqref{3.20}
depend only on $\phi$, rather than on $\P$, and
admit simple closed-form expressions. This property
considerably simplifies the calculation of the EIFs
of both dependence measures.

\begin{proposition}[Normalizing constants for continuous $Y$, Remark 3.3(b) in \cite{ansari2026quantifying}]
\label{prop:normalizing_constants}
Suppose that $Y$ is continuously distributed. Then
\[
    \alpha_\phi
    =
    \int_0^1
    (1-u)\{\phi(u)+\phi(-u)\}\,\d u,
    \qquad
    \beta_\phi
    =
    \frac{\phi(1)+\phi(-1)}6.
\]
In particular, for $\phi(u)=u^2$, $\alpha_\phi=1/6$ and $\beta_\phi=1/3$.
\end{proposition}

We next derive the EIFs of the two classes of
dependence measures in~\eqref{eq:rmd} 
under the
nonparametric model
\[
\mathcal P_{\mathrm{cont}}
:=
\{\P:\P_Y\text{ is continuous}\}.
\]
Since bounded-score perturbations preserve
continuity of the marginal distribution,
the tangent space at every
$\P\in\mathcal P_{\mathrm{cont}}$
is $L_0^2(\P)$.

\begin{theorem}[EIFs of regression-based measures of dependence]
\label{thm:eif_generalized_chatterjee}
Suppose that $\phi:[-1,1]\to\R$ is continuously
differentiable, convex, strictly convex at zero,
and satisfies $\phi(0)=0$. 

\noindent
{\rm (i)} For $\xi_\phi(\P)$, write
\[
    \Delta_t(x):=Q_t(x)-F_Y(t),
    \qquad
    \overline\phi_\P(t)
    :=
    \int_{\cX}
    \phi'\{\Delta_t(u)\}\,\d\P_X(u),
\]
and define
\begin{align*}
    A_{\phi,\P}^{\xi}(y)
    &:=
    \int_{\cX}
    \phi\{\Delta_y(u)\}\,\d\P_X(u),~~~
    B_{\phi,\P}^{\xi}(x)
    :=
    \int_{\R}
    \phi\{\Delta_t(x)\}\,\d\P_Y(t),\\
    C_{\phi,\P}^{\xi}(x,y)
    &:=
    \int_{\R}
    \phi'\{\Delta_t(x)\}
    \{\ind(y\leq t)-Q_t(x)\}
    \,\d\P_Y(t),\\
    D_{\phi,\P}^{\xi}(y)
    &:=
    -\int_{\R}
    \overline\phi_\P(t)
    \{\ind(y\leq t)-F_Y(t)\}
    \,\d\P_Y(t).
\end{align*}
Then $\xi_\phi$ is differentiable at
$\P\in \mathcal P_{\mathrm{cont}}$, with EIF
\begin{equation}
\begin{aligned}
    \widetilde\xi_{\phi,\P}(x,y)
    =
    \frac1{\alpha_\phi}\Big\{
        &A_{\phi,\P}^{\xi}(y)
        +B_{\phi,\P}^{\xi}(x)
        +C_{\phi,\P}^{\xi}(x,y)
        +D_{\phi,\P}^{\xi}(y)
        -2\alpha_\phi\xi_\phi(\P)
    \Big\}.
\end{aligned}
\label{eq:eif_xi_phi}
\end{equation}

\noindent
{\rm (ii)} For $\Lambda_\phi(\P)$, define
\begin{align*}
    A_{\phi,\P}^{\Lambda}(y)
    &:=
    \int_{\cX^2}
    \phi\{Q_y(u)-Q_y(v)\}
    \,\d\P_X(u)\d\P_X(v),\\
    B_{\phi,\P}^{\Lambda}(x)
    &:=
    \int_{\R}\int_{\cX}
    \Big[
        \phi\{Q_t(x)-Q_t(u)\}
        +\phi\{Q_t(u)-Q_t(x)\}
    \Big]
    \,\d\P_X(u)\d\P_Y(t),\\
    C_{\phi,\P}^{\Lambda}(x,y)
    &:=
    \int_{\R}\int_{\cX}
    \Big[
        \phi'\{Q_t(x)-Q_t(u)\}
        -\phi'\{Q_t(u)-Q_t(x)\}
    \Big]
    \{\ind(y\leq t)-Q_t(x)\}
    \,\d\P_X(u)\d\P_Y(t).
\end{align*}
Then $\Lambda_\phi$ is differentiable
at $\P\in \mathcal P_{\mathrm{cont}}$, with EIF
\begin{equation}
\begin{aligned}
    \widetilde\Lambda_{\phi,\P}(x,y)
    =
    \frac1{\beta_\phi}\Big\{
        &A_{\phi,\P}^{\Lambda}(y)
        +B_{\phi,\P}^{\Lambda}(x)+C_{\phi,\P}^{\Lambda}(x,y)
        -3\beta_\phi\Lambda_\phi(\P)
    \Big\}.
\end{aligned}
\label{eq:eif_lambda_phi}
\end{equation}
\end{theorem}

Theorem~\ref{thm:eif_generalized_chatterjee}
recovers, as a special case, the EIF of the
 Dette--Siburg--Stoimenov measure,
previously characterized by
\citet[Proposition 5.4]{klein2024asymptotic}
through its equivalent Cram\'er--von Mises
sensitivity index.

\begin{corollary}[EIF of the Dette--Siburg--Stoimenov measure]
\label{cor:eif_chatterjee}
Under the conditions of
Theorem~\ref{thm:eif_generalized_chatterjee},
let $\phi(u)=u^2$ and define
\[
    H_\P(y)
    :=
    \int_{\cX}Q_y(x)^2\,\d\P_X(x)~~~{\rm and}~~~
    \Psi_\P
    :=
    \int_{\R}H_\P(t)\,\d\P_Y(t).
\]
Then for any $\P\in \mathcal P_{\mathrm{cont}}$, $\xi_\phi(\P)
    =
    \Lambda_\phi(\P)
    =
    \xi(\P)
    =
    6\Psi_\P-2$,
and the corresponding EIFs coincide:
\begin{equation}
\begin{aligned}
    \widetilde\xi_\P(x,y)
    =
    6\Big[
        &H_\P(y)
        +\int_{\R}
        \Big\{
            2Q_t(x)\ind(y\leq t)
            -Q_t(x)^2
        \Big\}
        \,\d\P_Y(t)-2\Psi_\P
    \Big].
\end{aligned}
\label{eq:eif_chatterjee}
\end{equation}
\end{corollary}

We conclude this section by examining the behavior of
the EIFs at the two extremes of functional dependence.
As the following corollary shows, both EIFs vanish
under independence and perfect functional dependence.
This degeneracy is consistent with the behavior of
classical dependence measures discussed in
Remark~\ref{rem:classical_degeneracy}; see also
 \cite{chatterjee2020new}.

\begin{corollary}[Vanishing EIFs at the extremes of dependence]
\label{prop:vanishing_eif_independence}
Suppose that the conditions of
Theorem~\ref{thm:eif_generalized_chatterjee}
hold. If $X$ and $Y$ are independent, or if
$Y=f(X)$ almost surely for some measurable
function $f$, then
\[
    \widetilde\xi_{\phi,\P}(X,Y)
    =
    \widetilde\Lambda_{\phi,\P}(X,Y)
    =
    0,
    \qquad \P\text{-almost surely}.
\]
\end{corollary}

\section{Estimation of regression-based measures of dependence}
\label{sec:estimation}

\subsection{Rank- and graph-based estimators}
\label{sec:rank_graph_estimators}

Having characterized the EIFs of regression-based measures
of dependence and the corresponding semiparametric efficiency
bounds, we now turn to their estimation. For $\xi(\P)$, two particularly influential estimators
are Chatterjee's rank correlation coefficient
\citep{chatterjee2020new} and its NN graph-based extension
due to Azadkia and Chatterjee \citep{azadkia2019simple}.
Both estimators are valued for being tuning-parameter-free, relying only on
response ranks and simple local comparisons, and being
computationally efficient
\citep{han2021extensions,chatterjee2024survey}.
As we show below, however, neither generally attains
the semiparametric efficiency bound.

We begin with Chatterjee's rank correlation coefficient
for a scalar covariate $X\in\R$. Let
$(X_1,Y_1),\ldots,(X_n,Y_n)$ be independent copies
of $(X,Y)$, and suppose that both $X$ and $Y$ are
continuously distributed. Let $\pi$ be the permutation
of $\{1,\ldots,n\}$ such that
\[
    X_{\pi(1)}<\cdots<X_{\pi(n)},
\]
and define the response ranks
\[
    R_i:=\sum_{j=1}^n\ind(Y_j\leq Y_i),
    \qquad
    r_i:=R_{\pi(i)}.
\]
Chatterjee's rank correlation coefficient is then
given by
\begin{equation}
    \widehat\xi_n^{\,\mathrm{Ch}}
    :=
    1-\frac{3}{n^2-1}
    \sum_{i=1}^{n-1}|r_{i+1}-r_i|.
    \label{eq:crc}
\end{equation}
Azadkia and Chatterjee proposed a graph-based extension
of~\eqref{eq:crc} to a multivariate covariate
$X\in\cX\subseteq\R^d$. For each
$i\in\{1,\ldots,n\}$, let $N(i)$ denote the index
of an NN of $X_i$ among $\{X_j:j\neq i\}$,
under the Euclidean metric. Their unconditional
graph correlation coefficient is defined by
\begin{equation}
    \widehat\xi_n^{\,\mathrm{AC}}
    :=
    \frac{6}{n^2-1}
    \sum_{i=1}^n
    \min\{R_i,R_{N(i)}\}
    -
    \frac{2n+1}{n-1}.
    \label{eq:ac}
\end{equation}
When $d=1$, both~\eqref{eq:crc} and~\eqref{eq:ac}
estimate the same population parameter $\xi(\P)$.
Nevertheless, they are distinct statistics and generally
have different asymptotic variances.

We next summarize recently obtained closed-form
expressions for their limiting variances. Throughout,
we retain the notation
\[
    Q_t(x)=\P(Y\leq t\mid X=x),
    \qquad
    H_\P(t)=\E_\P[Q_t(X)^2],
\]
introduced in Corollary~\ref{cor:eif_chatterjee}.
Since $H_\P$ is nondecreasing, we write
$\d H_\P(t)$ for its associated Lebesgue--Stieltjes
measure.

For Chatterjee's rank correlation coefficient, define,
for $s,t\in\R$,
\begin{align*}
    C_{11}(t,s)
    &:=
    \E_\P[Q_t(X)^2Q_s(X)^2]
    -H_\P(t)H_\P(s)\\
    &\quad+
    \E_\P\Big[
        \{Q_{t\wedge s}(X)+3Q_t(X)Q_s(X)\}
        \{Q_{t\wedge s}(X)-Q_t(X)Q_s(X)\}
    \Big],\\
C_{12}(t,s)
&:=\E_\P[Q_t(X)^2Q_s(X)]
  -H_\P(t)F_Y(s)
+2\E_\P\!\left[
  Q_t(X)\{Q_{t\wedge s}(X)-Q_t(X)Q_s(X)\}
\right],\\
    C_{22}(t,s)
    &:=
    F_Y(t\wedge s)-F_Y(t)F_Y(s).
\end{align*}

For the graph-based estimator, let
$\widetilde Y,\widetilde Y',\widetilde Y''$
be conditionally independent copies of $Y$ given $X$,
also conditionally independent of $Y$, with common
conditional distribution $\P_{Y\mid X}$.
In addition, let $(X_1,Y_1)$ and $(X_2,Y_2)$
be independent copies of $(X,Y)$. For each
$i\in\{1,2\}$, let $\widetilde Y_i$ and
$\widetilde Y_i'$ be conditionally independent copies
of $Y_i$ given $X_i$, also conditionally independent
of $Y_i$. Define
\begin{align*}
    T_1
    &:=
    \E_\P[
        F_Y^2(Y\wedge\widetilde Y)
    ],~~~
    T_2
    :=
    \E_\P[
        F_Y(Y\wedge\widetilde Y)
        F_Y(Y\wedge\widetilde Y')
    ],\\
    T_3
    &:=
    \E_\P[
        F_Y(Y\wedge\widetilde Y)
        F_Y(\widetilde Y'\wedge\widetilde Y'')
    ],~~~
    T_4
    :=
    \E_\P[
        \ind(Y_1\leq Y_2\wedge\widetilde Y_2)
        F_Y(Y_1\wedge\widetilde Y_1)
    ],\\
    T_5
    &:=
    \E_\P[
        \ind(Y_1\leq Y_2\wedge\widetilde Y_2)
        F_Y(\widetilde Y_1\wedge\widetilde Y_1')
    ],~~~
    T_6
    :=
    \E_\P\Big[
        F_Y(
            Y_1\wedge\widetilde Y_1
            \wedge Y_2\wedge\widetilde Y_2
        )
    \Big],\\
    T_7
    &:=
    \Big\{
        \E_\P[
            F_Y(Y\wedge\widetilde Y)
        ]
    \Big\}^2.
\end{align*}
Finally, let $\mathfrak{q}_d$ and $\mathfrak{o}_d$ denote the two positive
geometric constants associated with the
$d$-dimensional NN graph
\citep[Theorem 3.1]{shi2021ac}, given by
\begin{align*}
    \mathfrak{q}_d
    &:=
    \left\{
        2-
        I_{3/4}\left(
            \frac{d+1}{2},\frac12
        \right)
    \right\}^{-1},~~~
    \mathfrak{o}_d
    :=
    \int_{\Gamma_{d;2}}
    \exp\left[
        -\lambda\left\{
            B(w_1,\|w_1\|)
            \cup
            B(w_2,\|w_2\|)
        \right\}
    \right]
    \d(w_1,w_2),
\end{align*}
where
\[
    \Gamma_{d;2}
    :=
    \left\{
        (w_1,w_2)\in(\R^d)^2:
        \max\{\|w_1\|,\|w_2\|\}
        <
        \|w_1-w_2\|
    \right\},
\]
$B(w,r)$ denotes the Euclidean ball centered at $w$
with radius $r$, $\lambda$ denotes Lebesgue measure
on $\R^d$, and $I_x(a,b)$ is the regularized
incomplete beta function.

The following theorem summarizes the closed-form asymptotic variances derived by \citet[Theorem 2.3]{chhaibi2026martingale} and \citet[Theorem 2.1]{gao2026limit}, respectively. For Chatterjee's rank correlation coefficient, we correct a transposition of the off-diagonal covariance terms in \citet[Theorem 2.3]{chhaibi2026martingale}, ensuring that $C_{12}(t,s)$ is indexed consistently with the corresponding integration measures. Combined with the earlier CLTs of \citet{lin2022limit} (see also \citealp{kroll2024asymptotic}), these results yield the following non-studentized asymptotic distributions.

\begin{theorem}[Closed-form limiting variances]
\label{thm:rank_graph_variance}
Suppose that $Y$ is continuously distributed.
\begin{enumerate}[itemsep=-.5ex,label=(\roman*)]
    \item Suppose that $d=1$ and $X$ is
    continuously distributed. Under conditions
    ensuring that the bias of
    $\widehat\xi_n^{\,\mathrm{Ch}}$ is
    $o(n^{-1/2})$ (cf.\
    \citet[Proposition 1.1]{lin2022limit}),
    \[
        \sqrt n
        \{\widehat\xi_n^{\,\mathrm{Ch}}-\xi(\P)\}
        \Rightarrow
        \cN(0,\sigma_{\mathrm{Ch}}^2),
    \]
    where
    \begin{align*}
        \sigma_{\mathrm{Ch}}^2
        =36\Big\{
        &\int_{\R^2}
        C_{11}(t,s)
        \,\d\P_Y(t)\d\P_Y(s)
        +\int_{\R^2}
        C_{22}(t,s)
        \,\d H_\P(t)\d H_\P(s)
        -2\int_{\R^2}
        C_{12}(t,s)
        \,\d\P_Y(t)\d H_\P(s)
        \Big\}.
    \end{align*}

    \item Suppose that $d\geq 1$ and $X$ is Lebesgue absolutely
    continuous and has a continuous density on
    its support. Assume further that the bias of
    $\widehat\xi_n^{\,\mathrm{AC}}$ is
    $o(n^{-1/2})$, either directly or after
    a bias correction that preserves the
    first-order stochastic term
    \citep{azadkia2026bias}. Then
    \[
        \sqrt n
        \{\widehat\xi_n^{\,\mathrm{AC}}-\xi(\P)\}
        \Rightarrow
        \cN(0,\sigma_{\mathrm{AC},d}^2),
    \]
    where
    \begin{align*}
        \sigma_{\mathrm{AC},d}^2
        =36\Big\{
        &(1+\mathfrak{q}_d)T_1
        +(2-2\mathfrak{q}_d+\mathfrak{o}_d)T_2
        -(2-\mathfrak{q}_d+\mathfrak{o}_d)T_3
        +4T_4-2T_5+T_6-4T_7
        \Big\}.
    \end{align*}
\end{enumerate}
\end{theorem}

Under independence, these limiting variances
simplify to the distribution-free expressions
\[
    \sigma_{\mathrm{Ch}}^2=\frac25,
    \qquad
    \sigma_{\mathrm{AC},d}^2
    =
    \frac25+\frac25\mathfrak{q}_d+\frac45\mathfrak{o}_d,
\]
previously derived by
\cite{chatterjee2020new} and
\cite{shi2021ac},
respectively. In particular, both estimators
have nondegenerate root-$n$ fluctuations under
independence. By contrast,
Corollary~\ref{prop:vanishing_eif_independence}
shows that the EIF of $\xi(\P)$ vanishes
at independence, so the corresponding
semiparametric efficiency bound is zero.
Thus, neither estimator attains the
efficiency bound at independence.

We next extend this comparison to general
dependence structures. Specifically, we compare
the limiting variances of
$\widehat\xi_n^{\,\mathrm{Ch}}$ and
$\widehat\xi_n^{\,\mathrm{AC}}$ with
\[
    \mathcal V_{\mathrm{eff}}(\P;\xi)
    :=
    \E_\P[\widetilde\xi_\P(X,Y)^2],
\]
where $\widetilde\xi_\P$ is the EIF
given in~\eqref{eq:eif_chatterjee}.
To quantify the efficiency gaps, define,
for $s,t\in\R$ and $x\in\cX$,
\[
    R_\P(t,s;x)
    :=
    Q_{s\wedge t}(x)-Q_s(x)Q_t(x),
~~~
    \Delta_{1,\P}
    :=
    \int_{\R^2}
    \E_\P[R_\P(t,s;X)^2]
    \,\d\P_Y(t)\d\P_Y(s).
\]
Notice that $R_\P(t,s;X)$ is the conditional
covariance of $\ind(Y\leq t)$ and
$\ind(Y\leq s)$ given $X$.
Recall that $\widetilde Y$ denotes a
conditionally independent copy of $Y$ given
$X$. Define
\begin{align*}
    m_\P(x,y)
    &:=
    \E_\P[
        F_Y(y\wedge\widetilde Y)
        \mid X=x
    ],~~~
    \theta_\P(x)
    :=
    \E_\P[
        F_Y(Y\wedge\widetilde Y)
        \mid X=x
    ],
\end{align*}
and set
\[
    \Delta_{2,\P}
    :=
    \E_\P[
        \{m_\P(X,Y)-\theta_\P(X)\}^2
    ].
\]
Both $\Delta_{1,\P}$ and $\Delta_{2,\P}$
are nonnegative by construction.

The following theorem gives an explicit
decomposition of the limiting variances
into the semiparametric efficiency bound
and additional nonnegative terms.

\begin{theorem}[Efficiency gaps of rank- and graph-based estimators]
\label{thm:general_inefficiency}
Suppose that $Y$ is continuously distributed
and that the corresponding conditions of
Theorem~\ref{thm:rank_graph_variance}
hold for each estimator under consideration.
Then the following hold true.
\begin{enumerate}[itemsep=-.5ex,label=(\roman*)]
    \item The limiting variance of Chatterjee's
    rank correlation coefficient satisfies
    \[
        \sigma_{\mathrm{Ch}}^2
        =
        \mathcal V_{\mathrm{eff}}(\P;\xi)
        +36\Delta_{1,\P}.
    \]

    \item The limiting variance of the
    Azadkia--Chatterjee graph correlation
    coefficient satisfies
    \[
        \sigma_{\mathrm{AC},d}^2
        =
        \mathcal V_{\mathrm{eff}}(\P;\xi)
        +
        36\left\{
            (1+\mathfrak{q}_d)\Delta_{1,\P}
            +\mathfrak{o}_d\Delta_{2,\P}
        \right\}.
    \]

    \item $\Delta_{1,\P}>0$ and $\Delta_{2,\P}>0$ if and only if $Y$ is not almost surely a measurable function of $X$. Consequently, both estimators have asymptotic variances strictly exceeding the semiparametric efficiency bound, except under perfect functional dependence, where both limiting variances and the efficiency bound vanish.
\end{enumerate}
\end{theorem}

\subsection{Efficient estimation}
\label{sec:efficient_estimation}

We now construct efficient estimators of $\xi_\phi(\P)$ and $\Lambda_\phi(\P)$ in~\eqref{eq:rmd}, both of which reduce to $\xi(\P)$ when $\phi(u)=u^2$. Our construction follows the general cross-fitting strategy outlined in Section~\ref{sec:general-method}.

Recall that $\{\mathcal I_1,\ldots,\mathcal I_K\}$ is a deterministic partition of $\{1,\ldots,n\}$ into a fixed number $K\geq2$ of folds, with $n_k:=|\mathcal I_k|\asymp n$. For each fold $k$, we use only observations outside $\mathcal I_k$ to construct an estimator $\widehat Q_{k,t}(x)$ of $Q_t(x)$, truncated to $[0,1]$ if necessary. Let
\[
    \PP_{n,k}:=\frac1{n_k}\sum_{i\in\mathcal I_k}\delta_{Z_i}
\]
denote the empirical measure of the $k$th validation fold, with empirical marginals
\[
    \PP_{X,k}:=\frac1{n_k}\sum_{i\in\mathcal I_k}\delta_{X_i},
    \qquad
    \PP_{Y,k}:=\frac1{n_k}\sum_{i\in\mathcal I_k}\delta_{Y_i}.
\]
We also define the empirical cumulative distribution function (CDF) of the validation responses by
\[
    \widehat F_k(t):=\int\ind(y\leq t)\,\d\PP_{Y,k}(y).
\]
By construction, $\widehat Q_{k,t}$ is independent of the validation sample $\PP_{n,k}$, whereas $\widehat F_k$ and the empirical marginal measures are computed from that sample. As discussed earlier in Remark \ref{remark:general-dml} and Section \ref{sec:sobol-efficient-estimation}, this additional dependence is permitted in our construction and will be explicitly accounted for in the asymptotic analysis.

For the two classes of dependence measures, define the fold-specific plug-in numerators
\begin{align*}
    \widehat N^\xi_{\phi,k}
    &:=\int_{\R}\int_{\cX}\phi\{\widehat Q_{k,t}(x)-\widehat F_k(t)\}\,\d\PP_{X,k}(x)\d\PP_{Y,k}(t),\\
    \widehat N^\Lambda_{\phi,k}
    &:=\int_{\R}\int_{\cX^2}\phi\{\widehat Q_{k,t}(x_1)-\widehat Q_{k,t}(x_2)\}\,\d(\PP_{X,k}\otimes\PP_{X,k})(x_1,x_2)\d\PP_{Y,k}(t).
\end{align*}
The corresponding fold-specific plug-in estimators are
\[
    \widehat\xi_{\phi,k}^{\mathrm{plug}}:=\frac{\widehat N^\xi_{\phi,k}}{\alpha_\phi},
    \qquad
    \widehat\Lambda_{\phi,k}^{\mathrm{plug}}:=\frac{\widehat N^\Lambda_{\phi,k}}{\beta_\phi}.
\]

\paragraph*{The $\xi_\phi$ family.}
Motivated by the EIF in Theorem~\ref{thm:eif_generalized_chatterjee}(i), define
\begin{align*}
    \widehat A^\xi_{\phi,k}(y)
    &:=\int_{\cX}\phi\{\widehat Q_{k,y}(u)-\widehat F_k(y)\}\,\d\PP_{X,k}(u),\\
    \widehat B^\xi_{\phi,k}(x)
    &:=\int_{\R}\phi\{\widehat Q_{k,t}(x)-\widehat F_k(t)\}\,\d\PP_{Y,k}(t),\\
    \widehat C^\xi_{\phi,k}(x,y)
    &:=\int_{\R}\phi'\{\widehat Q_{k,t}(x)-\widehat F_k(t)\}
    \{\ind(y\leq t)-\widehat Q_{k,t}(x)\}\,\d\PP_{Y,k}(t),\\
    \widehat D^\xi_{\phi,k}(y)
    &:=-\int_{\R}\widehat J_{\phi,k}(t)
    \{\ind(y\leq t)-\widehat F_k(t)\}\,\d\PP_{Y,k}(t),
\end{align*}
where
\[
    \widehat J_{\phi,k}(t):=\int_{\cX}\phi'\{\widehat Q_{k,t}(u)-\widehat F_k(t)\}\,\d\PP_{X,k}(u).
\]
These quantities are empirical analogues of the corresponding components of the EIF. In particular, the use of the validation-fold empirical measures yields the exact identities
\[
    \PP_{n,k}\widehat A^\xi_{\phi,k}
    =\PP_{n,k}\widehat B^\xi_{\phi,k}
    =\widehat N^\xi_{\phi,k},
    \qquad
    \PP_{n,k}\widehat D^\xi_{\phi,k}=0.
\]
Consequently, the empirical centering terms cancel in the one-step construction, giving the cross-fitted estimator
\begin{align}
    \widehat\xi_{\phi,n}^{\mathrm{eff}}
    &:=\sum_{k=1}^K\frac{n_k}{n\alpha_\phi}
    \left\{\widehat N^\xi_{\phi,k}+\PP_{n,k}\widehat C^\xi_{\phi,k}\right\}=\frac1{n\alpha_\phi}\sum_{k=1}^K\sum_{i\in\mathcal I_k}
    \left\{\widehat N^\xi_{\phi,k}+\widehat C^\xi_{\phi,k}(X_i,Y_i)\right\}.
    \label{eq:xi_eff_explicit}
\end{align}

\paragraph*{The $\Lambda_\phi$ family.}
For the pairwise dependence measure, we similarly define
\begin{align*}
    \widehat A^\Lambda_{\phi,k}(y)
    &:=\int_{\cX^2}\phi\{\widehat Q_{k,y}(u)-\widehat Q_{k,y}(v)\}\,\d(\PP_{X,k}\otimes\PP_{X,k})(u,v),\\
    \widehat B^\Lambda_{\phi,k}(x)
    &:=\int_{\R}\int_{\cX}
    \Big[\phi\{\widehat Q_{k,t}(x)-\widehat Q_{k,t}(u)\}
    +\phi\{\widehat Q_{k,t}(u)-\widehat Q_{k,t}(x)\}\Big]
    \,\d\PP_{X,k}(u)\d\PP_{Y,k}(t),\\
    \widehat C^\Lambda_{\phi,k}(x,y)
    &:=\int_{\R}\int_{\cX}
    \Big[\phi'\{\widehat Q_{k,t}(x)-\widehat Q_{k,t}(u)\}
    -\phi'\{\widehat Q_{k,t}(u)-\widehat Q_{k,t}(x)\}\Big]\\
    &\qquad\qquad\times
    \{\ind(y\leq t)-\widehat Q_{k,t}(x)\}
    \,\d\PP_{X,k}(u)\d\PP_{Y,k}(t).
\end{align*}
As in the preceding construction, the empirical measures give rise to the identities
\[
    \PP_{n,k}\widehat A^\Lambda_{\phi,k}=\widehat N^\Lambda_{\phi,k},
    \qquad
    \PP_{n,k}\widehat B^\Lambda_{\phi,k}=2\widehat N^\Lambda_{\phi,k}.
\]
The corresponding cross-fitted one-step estimator therefore simplifies to
\begin{align}
    \widehat\Lambda_{\phi,n}^{\mathrm{eff}}
    &:=\sum_{k=1}^K\frac{n_k}{n\beta_\phi}
    \left\{\widehat N^\Lambda_{\phi,k}+\PP_{n,k}\widehat C^\Lambda_{\phi,k}\right\}=\frac1{n\beta_\phi}\sum_{k=1}^K\sum_{i\in\mathcal I_k}
    \left\{\widehat N^\Lambda_{\phi,k}+\widehat C^\Lambda_{\phi,k}(X_i,Y_i)\right\}.
    \label{eq:Lambda_eff_explicit}
\end{align}

\paragraph*{Chatterjee's coefficient.}
For the quadratic choice $\phi(u)=u^2$, the two population functionals coincide:
\[
    \xi_\phi(\P)=\Lambda_\phi(\P)=\xi(\P),
    \qquad
    \alpha_\phi=\frac16,\qquad \beta_\phi=\frac13.
\]
Recall from Corollary~\ref{cor:eif_chatterjee} that
\[
    \Psi_\P:=\int_{\R}\int_{\cX}Q_t(x)^2\,\d\P_X(x)\d\P_Y(t),
    \qquad
    \xi(\P)=6\Psi_\P-2.
\]
This alternative representation leads to a particularly simple estimator. For each fold $k$, define
\[
    \widehat H_k(t):=\int_{\cX}\widehat Q_{k,t}(x)^2\,\d\PP_{X,k}(x),
    \qquad
    \widehat\Psi_k:=\int_{\R}\widehat H_k(t)\,\d\PP_{Y,k}(t).
\]
The corresponding fold-specific plug-in estimator is
\[
    \widehat\xi_k^{\mathrm{plug}}:=6\widehat\Psi_k-2.
\]
Motivated by~\eqref{eq:eif_chatterjee}, we estimate the EIF by
\begin{align*}
    \widehat{\widetilde\xi}_k(x,y)
    :=6\Bigg[
    \widehat H_k(y)
    +\int_{\R}\left\{2\widehat Q_{k,t}(x)\ind(y\leq t)-\widehat Q_{k,t}(x)^2\right\}\,\d\PP_{Y,k}(t)
    -2\widehat\Psi_k
    \Bigg].
\end{align*}
Since $\PP_{n,k}\widehat H_k=\widehat\Psi_k$, the empirical centering terms cancel upon averaging the estimated EIF over the validation fold. Writing $k(i)$ for the fold containing observation $i$, the resulting cross-fitted one-step estimator is
\begin{align}
    \widehat\xi_n^{\mathrm{eff}}
    &:=\frac1n\sum_{i=1}^n
    \left\{\widehat\xi_{k(i)}^{\mathrm{plug}}
    +\widehat{\widetilde\xi}_{k(i)}(X_i,Y_i)\right\}\notag\\
    &=\frac6n\sum_{i=1}^n\int_{\R}
    \left\{2\widehat Q_{k(i),t}(X_i)\ind(Y_i\leq t)
    -\widehat Q_{k(i),t}(X_i)^2\right\}\,\d\PP_{Y,k(i)}(t)-2\notag\\
    &=\frac6n\sum_{k=1}^K\frac1{n_k}\sum_{i,j\in\mathcal I_k}
    \left\{2\widehat Q_{k,Y_j}(X_i)\ind(Y_i\leq Y_j)
    -\widehat Q_{k,Y_j}(X_i)^2\right\}-2.
    \label{eq:xi_eff_double_sum}
\end{align}
Thus, conditional on the training sample for each fold, its contribution to $\widehat\xi_n^{\mathrm{eff}}$ is a second-order $V$-statistic based on the corresponding validation observations. This representation will be particularly useful for establishing the asymptotic linearity of the estimator.

Although the general estimators in~\eqref{eq:xi_eff_explicit}--\eqref{eq:Lambda_eff_explicit} and the specialized estimator in~\eqref{eq:xi_eff_double_sum} target the same population coefficient when $\phi(u)=u^2$, they need not coincide in finite samples. Nevertheless, as shown below, all three estimators admit the same first-order asymptotic linear representation under suitable conditions.

\subsection{Efficiency theory}
\label{sec:efficient_estimation_theory}

We now establish the asymptotic efficiency of the cross-fitted estimators introduced in the preceding subsection. The key feature of our construction is that, for each fold $k$, the estimated conditional CDF $\widehat Q_k=\widehat Q_{k,\cdot}(\cdot)$ is independent of the validation sample $\PP_{n,k}$ and its empirical marginals $\PP_{X,k}$ and $\PP_{Y,k}$. This sample-splitting structure permits a conditional analysis without imposing Donsker conditions on the first-stage estimators. 

Such an approach is now widely used in causal inference and econometrics, particularly in the double/debiased machine learning framework \citep{chernozhukov2018double}. However, unlike standard cross-fitting constructions, our estimators also involve empirical marginal distributions computed from the validation folds. Their asymptotic analysis therefore requires accounting for the resulting within-fold dependence.

In detail, for $a\in\{\xi,\Lambda\}$, write
\[
    \theta_{\phi,a}(\P):=
    \begin{cases}
        \xi_\phi(\P), & ~~{\rm if}~a=\xi,\\
        \Lambda_\phi(\P), &  ~~{\rm if}~a=\Lambda,
    \end{cases}
\]
and let $\widetilde\theta_{\phi,a,\P}$ denote the corresponding EIF derived in Theorem~\ref{thm:eif_generalized_chatterjee}. For any measurable function $q=(q_t(x))_{x\in\cX,t\in\R}$ taking values in $[0,1]$, define
\[
    \|q-Q\|_{2,\P_X\otimes\P_Y}^2
    :=\int_{\R}\int_{\cX}\{q_t(x)-Q_t(x)\}^2\,\d\P_X(x)\d\P_Y(t).
\]

We impose the following assumptions on the sampling scheme, the function $\phi$, and the first-stage conditional-CDF estimators.

\begin{assumption}[Sampling and smoothness]
\label{ass:efficient_sampling}~~ \vspace{-.3cm}
\begin{enumerate}[itemsep=-.5ex,label=(\roman*)]
    \item The observations $Z_1,\ldots,Z_n$ are i.i.d.\ from $\P$.
    \item The marginal distribution of $Y$ is continuous.
    \item The function $\phi\in C^3([-1,1])$ is convex, strictly convex at zero, and satisfies $\phi(0)=0$.
\end{enumerate}
\end{assumption}

\begin{assumption}[Cross-fitting and first-stage estimation]
\label{ass:efficient_first_stage}
~~ \vspace{-.3cm}
\begin{enumerate}[itemsep=-.5ex,label=(\roman*)]
    \item The number of folds $K\geq2$ is fixed, and there exists a constant $c>0$ such that
    \[
        \min_{1\leq k\leq K}\frac{n_k}{n}\geq c
    \]
    for all sufficiently large $n$.

    \item For each fold $k$, the estimator $\widehat Q_k=\widehat Q_{k,\cdot}(\cdot)$ is measurable with respect to the observations $\{Z_i:i\notin\mathcal I_k\}$ and satisfies $0\leq\widehat Q_{k,t}(x)\leq1$. Moreover,
    \begin{equation}
        r_n:=\max_{1\leq k\leq K}\|\widehat Q_k-Q\|_{2,\P_X\otimes\P_Y}
        =o_\P(n^{-1/4}).
    \end{equation}
\end{enumerate}
\end{assumption}

The boundedness of the conditional-CDF estimators and the smoothness of $\phi$ allow the fold-specific statistics to be analyzed through conditional $V$-statistic expansions with controlled remainders. In particular, the independence between $\widehat Q_k$ and the corresponding validation observations permits us to condition on the first-stage estimator, while the dependence arising from the empirical marginal distributions is handled through the associated first-order projections. Together with the quadratic remainder in the first-stage estimation error, this yields the following result.

\begin{theorem}[Efficiency of the cross-fitted estimators]
\label{thm:general_efficient_estimator}
Suppose that Assumptions~\ref{ass:efficient_sampling} and~\ref{ass:efficient_first_stage} hold. Then, for each $a\in\{\xi,\Lambda\}$, the corresponding cross-fitted estimator satisfies
\begin{align}
    \widehat\theta_{\phi,a,n}^{\mathrm{eff}}-\theta_{\phi,a}(\P)
    &=\frac1n\sum_{i=1}^n\widetilde\theta_{\phi,a,\P}(Z_i)
    +O_\P\left(r_n^2+\frac{r_n}{\sqrt n}+\frac1n\right)\notag\\
    &=\frac1n\sum_{i=1}^n\widetilde\theta_{\phi,a,\P}(Z_i)+o_\P(n^{-1/2});
    \label{eq:AL_two_general}
\end{align}
consequently, both cross-fitted estimators are efficient in view of Theorem \ref{thm:efficiency_characterization}.
\end{theorem}

The same argument applies to the specialized estimator of Chatterjee's coefficient in~\eqref{eq:xi_eff_double_sum}, yielding the following corollary.

\begin{corollary}[Efficient estimation of Chatterjee's coefficient]
\label{cor:efficient_xi}
Suppose that Assumptions~\ref{ass:efficient_sampling} and~\ref{ass:efficient_first_stage} hold with $\phi(u)=u^2$. Then the estimator $\widehat\xi_n^{\mathrm{eff}}$ defined in~\eqref{eq:xi_eff_double_sum} satisfies
\[
    \widehat\xi_n^{\mathrm{eff}}-\xi(\P)
    =\frac1n\sum_{i=1}^n\widetilde\xi_\P(X_i,Y_i)+o_\P(n^{-1/2}),
\]
and hence is efficient in view of Theorem \ref{thm:efficiency_characterization}.
\end{corollary}

\section{Discussion}
\label{sec:discussion}

We conclude by discussing three aspects of the proposed
methodology: statistical inference, independence testing,
and rank-based implementation.

\paragraph*{Statistical inference.}
The asymptotic linear representations established in this
paper provide a direct basis for statistical inference.
In particular, for a dependence measure $\theta(\P)$
with EIF $\widetilde\theta_\P$, the asymptotic variance
is given by
\[
\sigma_\P^2
=
\P\widetilde\theta_\P^2.
\]
A natural estimator of $\sigma_\P^2$ is the empirical
variance of the estimated influence functions,
evaluated on the corresponding validation folds.
Provided that these estimated influence functions are
consistent in the appropriate $L^2$ sense, the resulting
variance estimator is consistent, yielding
asymptotically valid Wald-type confidence intervals
whenever $\sigma_\P^2>0$ \citep{newey2018cross,chernozhukov2018double}.

Bootstrap inference provides another attractive
alternative. Recently, \cite{lin2026bootstrap} established
bootstrap consistency for a general class of
double/debiased machine learning estimators under
exchangeably weighted resampling schemes, including
Efron's bootstrap. We expect the corresponding
bootstrap consistency results to follow from suitable
extensions of the arguments developed in this paper
and those of \cite{lin2026bootstrap}.

\paragraph*{Independence testing.}
Although our primary focus is efficient estimation,
the proposed estimators can also be used to construct
tests of independence. Since the population dependence
measures considered here vanish under independence,
a natural procedure is to reject for large values
of the estimated dependence measure. Critical values
can be obtained through permutation by recomputing the complete
estimation procedure for each permutation, with the data folds
and any algorithmic randomness fixed independently of the responses.
Under independence, this gives finite-sample validity by exchangeability.

Nevertheless, semiparametric efficiency in estimation does not necessarily translate into favorable local power properties for independence testing. In particular, Chatterjee's original rank correlation has been shown to exhibit zero Pitman efficiency against broad classes of regular alternatives \citep{shi2020power,lin2021boosting}, and similar limitations have been identified for kernel-based correlations \citep{azadkia2026kernel}. Whether the proposed estimators exhibit analogous behavior remains an open question, whose resolution requires a separate analysis of their higher-order fluctuations under contiguous alternatives.

That said, recent proposals \citep{dong2025spectral,gao2026limit} seek to improve power by combining tests based on estimators of functional dependence with conventional independence tests that are sensitive to local alternatives in regular statistical models, such as those based on Hoeffding's $D$. The resulting combined tests can be calibrated to control type-I error either by deriving the joint asymptotic distribution of the constituent test statistics or by applying a conservative Bonferroni-type adjustment.

\paragraph*{Rank-based implementation.}
An additional appealing feature of our framework
is that, in the bivariate setting with continuous
marginal distributions, the proposed estimators
can be made fully rank-based.
Rather than estimating the conditional CDFs using the original observations, one
may construct the preliminary estimators entirely
from the normalized ranks of the covariates and
responses. To preserve the sample-splitting independence,
fold-specific rank transformations used in training should be
constructed from the training observations only and then applied
to the validation points. By ensuring that all tuning choices
are rank-invariant, the resulting
cross-fitted one-step estimator depends only on
the joint ranks of the observations. 
Related uses of rank transformations in matching and graph-based dependence estimation can be found in \citet{cattaneo2025rosenbaum} and \citet{tran2024rank}.

\section{Proofs}
\label{sec:proofs}

This section collects the proofs of the results stated in the main text.  Throughout the following, let $\|\cdot\|_{\mathrm{TV}}$ denote the total variation norm, and QMD stand for quadratic mean differentiability. Specifically, a parametric submodel $\{\P_\epsilon:\epsilon\in(-\delta,\delta)\}$ through $\P=\P_0$ is said to be QMD at $\epsilon=0$ with score $s\in L_0^2(\P)$ if
\[
\int\left(
\sqrt{p_\epsilon}-\sqrt p
-\frac{\epsilon}{2}s\sqrt p
\right)^2\,\d\mu=o(\epsilon^2),
\qquad {\rm as}~~\epsilon\to0,
\]
where $p_\epsilon$ and $p$ are the densities of $\P_\epsilon$ and $\P$, respectively, with respect to a common dominating measure $\mu$.

\subsection{Proofs for Section~\ref{sec:general-method}}

\begin{proof}[Proof of Theorem~\ref{thm:cf_onestep_efficiency}]

For each fold $k$, add and subtract $\P\widetilde\psi_{\QQ_k}$ and $(\PP_{n,k}-\P)\widetilde\psi_\P$ to obtain
\begin{align*}
&\psi(\QQ_k)+\PP_{n,k}\widetilde\psi_{\QQ_k}-\psi(\P)=(\PP_{n,k}-\P)\widetilde\psi_\P
 +(\PP_{n,k}-\P)
       (\widetilde\psi_{\QQ_k}-\widetilde\psi_\P)
 +\Rem(\QQ_k,\P),
\end{align*}
so that
\begin{align*}
\widehat\psi_n^{**}-\psi(\P)
=&(\PP_n-\P)\widetilde\psi_\P+\sum_{k=1}^K\frac{n_k}{n}(\PP_{n,k}-\P)
    (\widetilde\psi_{\QQ_k}-\widetilde\psi_\P)
 +\sum_{k=1}^K\frac{n_k}{n}\Rem(\QQ_k,\P).
\end{align*}

Conditional on the training observations outside $\mathcal I_k$, the function $\widetilde\psi_{\QQ_k}-\widetilde\psi_\P$ is fixed, and the observations in $\mathcal I_k$ are independent with distribution $\P$. Consequently,
\begin{align*}
&\E\!\left[
 \left\{(\PP_{n,k}-\P)
 (\widetilde\psi_{\QQ_k}-\widetilde\psi_\P)\right\}^2
 \Bigm|\{Z_i:i\notin\mathcal I_k\}\right]=\frac1{n_k}\Var_\P
 (\widetilde\psi_{\QQ_k}-\widetilde\psi_\P)
\le\frac1{n_k}
 \|\widetilde\psi_{\QQ_k}-\widetilde\psi_\P\|_{L^2(\P)}^2.
\end{align*}
Because $K$ is fixed, $n_k\asymp n$, and $\|\widetilde\psi_{\QQ_k}-\widetilde\psi_\P\|_{L^2(\P)}=o_\P(1)$, conditional Chebyshev's inequality shows that each empirical remainder is $o_\P(n^{-1/2})$.

Finally, by assumption, $\max_k|\Rem(\QQ_k,\P)|=o_\P(n^{-1/2})$. Since $K$ is fixed, summing the foldwise bounds gives
\[
\widehat\psi_n^{**}-\psi(\P)
=(\PP_n-\P)\widetilde\psi_\P+o_\P(n^{-1/2})
=\frac1n\sum_{i=1}^n\widetilde\psi_\P(Z_i)
 +o_\P(n^{-1/2}).
\]
Invoking Theorem \ref{thm:efficiency_characterization} then completes the proof.
\end{proof}

\begin{proof}[Proof of Theorem~\ref{thm:sobol-efficiency}]

Fix a fold $k$, condition on the observations outside $\mathcal I_k$, and write $\delta_k(x)=\widehat m_k(x)-m(x)$. Expanding the square and collecting terms yields
\begin{align*}
2\widehat m_k(X)Y-\widehat m_k(X)^2
&=2\{m(X)+\delta_k(X)\}Y-\{m(X)+\delta_k(X)\}^2\\
&=2m(X)Y-m(X)^2
  +2\delta_k(X)\{Y-m(X)\}-\delta_k(X)^2.
\end{align*}
Since $m(X)=\E[Y\mid X]$ and $\delta_k$ is fixed under the conditioning, we have
\[
\P[\delta_k(X)\{Y-m(X)\}]
=\P\big[\delta_k(X)\E\{Y-m(X)\mid X\}\big]=0.
\]
Moreover, $\E[Y^4]<\infty$ implies $\E[\{Y-m(X)\}^4]<\infty$ by conditional Jensen's inequality. By the conditional variance formula and Cauchy--Schwarz,
\begin{align*}
\E\!\left[
 \left\{\PP_{n,k}\delta_k(X)\{Y-m(X)\}\right\}^2
 \Bigm|\{Z_i:i\notin\mathcal I_k\}\right]&\le\frac1{n_k}\P[\delta_k(X)^2\{Y-m(X)\}^2]\\
&\le\frac1{n_k}\|\delta_k\|_{L^4(\P_X)}^2
 \|Y-m(X)\|_{L^4(\P)}^2.
\end{align*}
Similarly,
\begin{align*}
\P\delta_k(X)^2=\|\delta_k\|_{L^2(\P_X)}^2\quad\text{and}\quad
\E\!\left[
 \left\{(\PP_{n,k}-\P)\delta_k(X)^2\right\}^2
 \Bigm|\{Z_i:i\notin\mathcal I_k\}\right]
\le\frac1{n_k}\|\delta_k\|_{L^4(\P_X)}^4.
\end{align*}
The $L^4$ consistency and $L^2$ rate assumptions, together with conditional Chebyshev's inequality (using truncation to handle the random norms), imply
\[
\PP_{n,k}[\delta_k(X)\{Y-m(X)\}]
=o_\P(n^{-1/2}),\qquad
\PP_{n,k}\delta_k(X)^2=o_\P(n^{-1/2}).
\]
Since $K$ is fixed and $n_k\asymp n$, summing the exact quadratic identity over the folds gives
\begin{align*}
\sum_{k=1}^K\frac{n_k}{n}\PP_{n,k}
\{2\widehat m_k(X)Y-\widehat m_k(X)^2\}
&=\P\{2m(X)Y-m(X)^2\}
 +(\PP_n-\P)\{2m(X)Y-m(X)^2\}
 +o_\P(n^{-1/2})\\
&=\P m(X)^2
 +(\PP_n-\P)\{2m(X)Y-m(X)^2\}
 +o_\P(n^{-1/2}).
\end{align*}

Next, since $\widehat\mu_k=\PP_{n,k}Y$ and $\widehat\mu=\PP_nY=\sum_k(n_k/n)\widehat\mu_k$, we have
\begin{align*}
\sum_{k=1}^K\frac{n_k}{n}\widehat\mu_k^2
&=\widehat\mu^2+
 \sum_{k=1}^K\frac{n_k}{n}
 (\widehat\mu_k-\widehat\mu)^2=\mu^2+2\mu(\PP_n-\P)Y+O_\P(n^{-1}).
\end{align*}
The last equality follows because $\widehat\mu_k-\mu=O_\P(n_k^{-1/2})$ and $\widehat\mu-\mu=O_\P(n^{-1/2})$. For the pooled variance estimator, similarly,
\begin{align*}
\widehat v-v
&=\PP_n(Y-\widehat\mu)^2-\P(Y-\mu)^2=(\PP_n-\P)(Y-\mu)^2-(\widehat\mu-\mu)^2=(\PP_n-\P)(Y-\mu)^2+O_\P(n^{-1}).
\end{align*}
Because $v>0$, we then obtain
\[
\frac1{\widehat v}
=\frac1v-\frac{\widehat v-v}{v^2}
  +o_\P(n^{-1/2}).
\]
Substituting these expansions into~\eqref{eq:sobol-eff-final} gives
\begin{align*}
\widehat\rho_{\mathrm{Sob},n}^{\mathrm{eff}}
-\rho_{\mathrm{Sob}}(\P)
=\frac1v(\PP_n-\P)\Big[
&2m(X)Y-m(X)^2-2\mu Y-\rho_{\mathrm{Sob}}(\P)(Y-\mu)^2
\Big]+o_\P(n^{-1/2}).
\end{align*}
Observe that
\begin{align*}
&2m(x)y-m(x)^2-2\mu y
 -\rho_{\mathrm{Sob}}(\P)(y-\mu)^2\\
 =&2\{m(x)-\mu\}\{y-m(x)\}
  +\{m(x)-\mu\}^2
  -\rho_{\mathrm{Sob}}(\P)(y-\mu)^2-\mu^2.
\end{align*}
The constant $-\mu^2$ disappears under $\PP_n-\P$. It follows that
\[
\widehat\rho_{\mathrm{Sob},n}^{\mathrm{eff}}-\rho_{\mathrm{Sob}}(\P)
=\frac1n\sum_{i=1}^n
 \widetilde\rho_{\mathrm{Sob},\P}(X_i,Y_i)+o_\P(n^{-1/2}),
\]
as claimed. 
\end{proof}

\subsection{Proofs for Section~\ref{sec:classical}}

\begin{proof}[Proof of Proposition~\ref{prop:u_eif}]

Fix a bounded score $s\in L_0^2(\P)$ and consider $\d\P_\epsilon=(1+\epsilon s)\d\P$. Expanding along this submodel gives
\begin{align*}
\psi_h(\P_\epsilon)
&=\E\!\left[h(Z_1,\ldots,Z_m)
 \prod_{j=1}^m\{1+\epsilon s(Z_j)\}\right]=\psi_h(\P)
 +\epsilon\sum_{j=1}^m
 \E\{h(Z_1,\ldots,Z_m)s(Z_j)\}+O(\epsilon^2).
\end{align*}
The $O(\epsilon^2)$ remainder follows from the finite product expansion: each higher-order term contains at least two factors of $\epsilon$, and is integrable by Cauchy--Schwarz, since $h\in L^2(\P^{\otimes m})$ and $s$ is bounded. By symmetry of $h$, each of the $m$ terms in the sum is equal to
\begin{align*}
\E\{h(Z_1,\ldots,Z_m)s(Z_1)\}
&=\E\!\left[s(Z_1)\E\{h(Z_1,Z_2,\ldots,Z_m)\mid Z_1\}\right]=\P\!\left[s\{h_{1,\P}+\psi_h(\P)\}\right]
=\P[s h_{1,\P}],
\end{align*}
where the last equality uses $\P s=0$. Thus
\[
\left.\frac{\d}{\d\epsilon}\psi_h(\P_\epsilon)\right|_{\epsilon=0}
=\P\{m h_{1,\P}s\}.
\]
By iterated expectation, $\P h_{1,\P}=0$, and conditional Jensen's inequality gives
\[
\P h_{1,\P}^2
\le\E h(Z_1,\ldots,Z_m)^2<\infty.
\]
By the assumed pathwise differentiability of $\psi_h$, its EIF $\widetilde\psi_{h,\P}\in L_0^2(\P)$ exists. In particular, for every bounded mean-zero score $s$, the defining property of the EIF and the derivative identity established above imply
\[
\P(\widetilde\psi_{h,\P}s)
=
\left.\frac{\d}{\d\epsilon}\psi_h(\P_\epsilon)\right|_{\epsilon=0}
=
\P(mh_{1,\P}s),
\]
and accordingly
\[
\P\left[\{\widetilde\psi_{h,\P}-mh_{1,\P}\}s\right]=0
\]
for every bounded $s\in L_0^2(\P)$. Since bounded mean-zero functions are dense in $L_0^2(\P)$ and both $\widetilde\psi_{h,\P}$ and $mh_{1,\P}$ belong to this space, the preceding identity extends by continuity to every $s\in L_0^2(\P)$, and thus $\widetilde\psi_{h,\P}(z)=mh_{1,\P}(z)$ almost surely, as claimed.
\end{proof}

\begin{proof}[Proof of Proposition~\ref{prop:u_efficiency}]

The Hoeffding decomposition of the fixed-order symmetric kernel $h$ yields
\[
U_n(h)-\psi_h(\P)
=\frac{m}{n}\sum_{i=1}^n h_{1,\P}(Z_i)
 +\sum_{j=2}^m {m\choose j}U_{n,j}(h_{j,\P})
=\frac1n\sum_{i=1}^n\widetilde\psi_{h,\P}(Z_i)
 +o_\P(n^{-1/2}),
\]
and the assertion follows from Theorem \ref{thm:efficiency_characterization}.
\end{proof}

\begin{proof}[Proof of Corollary~\ref{cor:classical_efficiency}]
(i) The population versions of Kendall's tau, Spearman's rho, Hoeffding's $D$, Blum--Kiefer--Rosenblatt's $R$, and Bergsma--Dassios--Yanagimoto's $\tau^*$ admit representations of the form~\eqref{eq:u-md}, with bounded symmetric kernels of fixed orders. We show that the boundedness of these kernels guarantees pathwise differentiability in the unrestricted nonparametric model.

Let $h$ be any bounded symmetric kernel of order $m$, and consider an arbitrary one-dimensional QMD submodel $\{\P_\epsilon\}$ through $\P$, with score $s\in L_0^2(\P)$. QMD implies the signed-measure expansion
\[
    \|\P_\epsilon-\P-\epsilon s\P\|_{\mathrm{TV}}=o(|\epsilon|),
\]
where $s\P$ denotes the signed measure with density $s$ with respect to $\P$. Since $m$ is fixed, the corresponding product measures satisfy
\[
    \P_\epsilon^{\otimes m}
    =
    \P^{\otimes m}
    +
    \epsilon\sum_{j=1}^m
    \P^{\otimes(j-1)}\otimes(s\P)\otimes\P^{\otimes(m-j)}
    +
    o_{\mathrm{TV}}(|\epsilon|).
\]
Because $h$ is bounded, integration against $h$ preserves this expansion. Consequently,
\[
\begin{aligned}
    \psi_h(\P_\epsilon)-\psi_h(\P)
    =
    \epsilon\sum_{j=1}^m
    \E_\P\{h(Z_1,\ldots,Z_m)s(Z_j)\}
    +o(|\epsilon|)=
    \epsilon m\,\P(h_{1,\P}s)+o(|\epsilon|),
\end{aligned}
\]
where the second equality follows from the symmetry of $h$ and $\P s=0$. Since $h$ is bounded, $mh_{1,\P}\in L_0^2(\P)$. Thus, the derivative identity holds along every QMD submodel and defines a bounded linear functional of its score. It follows that $\psi_h$ is pathwise differentiable at $\P$. By Proposition~\ref{prop:u_eif}, its EIF is $mh_{1,\P}$.

(ii) The population versions of squared distance covariance and HSIC admit the U-statistic representations. Under the assumed square-integrability of their kernels and pathwise differentiability of their population functionals, Proposition~\ref{prop:u_eif} identifies their EIFs as the corresponding first-order Hoeffding projections multiplied by the kernel orders. Proposition~\ref{prop:u_efficiency} then implies that their unbiased U-statistic estimators are semiparametrically efficient.

Finally, when the kernels defining HSIC are bounded, its associated fourth-order U-statistic kernel is also bounded. The argument used in part~(i) therefore establishes pathwise differentiability automatically, while square-integrability is immediate. This completes the proof.
\end{proof}

\subsection{Proofs for Section~\ref{sec:conditional_distribution}}

\begin{proof}[Proof of Proposition~\ref{prop:normalizing_constants}]
Let $Y_0,Y_1,Y_2$ be independent with marginal law $\P_Y$. Because $F_Y$ is continuous, $U=F_Y(Y_0)$ is uniform on $[0,1]$. Conditionally on $Y_0$ (and hence on $U=u$), $\ind(Y_1\le Y_0)$ is Bernoulli with success probability $u$. Consequently,
\begin{align*}
\alpha_\phi
&=\E\phi\{\ind(Y_1\le Y_0)-F_Y(Y_0)\}=\int_0^1\{u\phi(1-u)+(1-u)\phi(-u)\}\,\d u.
\end{align*}
Changing variables $v=1-u$ in the first summand gives
\[
\int_0^1u\phi(1-u)\,\d u
=\int_0^1(1-v)\phi(v)\,\d v,
\]
and hence
\[
\alpha_\phi
=\int_0^1(1-u)\{\phi(u)+\phi(-u)\}\,\d u.
\]
For $\beta_\phi$, conditionally on $U=u$, the two indicators $\ind(Y_1\le Y_0)$ and $\ind(Y_2\le Y_0)$ are independent Bernoulli$(u)$ variables. Their difference is $1$ with probability $u(1-u)$, is $-1$ with the same probability, and is $0$ otherwise. Using $\phi(0)=0$ gives
\begin{align*}
\beta_\phi
&=\int_0^1u(1-u)\{\phi(1)+\phi(-1)\}\,\d u=\frac{\phi(1)+\phi(-1)}6.
\end{align*}
For $\phi(u)=u^2$, these formulas give $\alpha_\phi=2\int_0^1(1-u)u^2\,\d u=1/6$ and $\beta_\phi=2/6=1/3$.
\end{proof}

\begin{proof}[Proof of Theorem~\ref{thm:eif_generalized_chatterjee}]
We first consider a bounded score $s\in L_0^2(\P)$ and the corresponding submodel $\d\P_\epsilon=(1+\epsilon s)\d\P$ for sufficiently small $|\epsilon|$.  Write $Q_{t,\epsilon}(x)$ and $F_{Y,\epsilon}(t)$ for the conditional and marginal CDFs under $\P_\epsilon$.

We first compute the derivatives of the conditional and marginal CDFs. By conditioning on $X$,
\begin{align*}
Q_{t,\epsilon}(x)
&=
\frac{\E[\ind(Y\le t)(1+\epsilon s(Z))\mid X=x]}
     {\E[1+\epsilon s(Z)\mid X=x]}=
\frac{Q_t(x)+\epsilon\E[\ind(Y\le t)s(Z)\mid X=x]}
     {1+\epsilon\E[s(Z)\mid X=x]}.
\end{align*}
Subtracting $Q_t(x)$ and dividing by $\epsilon$ gives
\[
\frac{Q_{t,\epsilon}(x)-Q_t(x)}{\epsilon}
=
\frac{\E[\{\ind(Y\le t)-Q_t(x)\}s(Z)\mid X=x]}
     {1+\epsilon\E[s(Z)\mid X=x]}.
\]
Hence
\begin{equation*}
\dot Q_t(x)
=
\E[\{\ind(Y\le t)-Q_t(x)\}s(Z)\mid X=x].
\end{equation*}
Similarly,
\begin{align*}
F_{Y,\epsilon}(t)
&=\P[(1+\epsilon s(Z))\ind(Y\le t)]=F_Y(t)+\epsilon\P[\{\ind(Y\le t)-F_Y(t)\}s(Z)],
\end{align*}
where the last equality uses $\P s=0$.  Therefore,
\[
\dot F_Y(t)
=\P[\{\ind(Y\le t)-F_Y(t)\}s(Z)].
\]
Finally, for any bounded measurable $f$ and $g$,
\begin{align*}
\left.\frac{\d}{\d\epsilon}\P_\epsilon f(X)\right|_{\epsilon=0}
&=\P[f(X)s(Z)]\qquad\text{and}\qquad\left.\frac{\d}{\d\epsilon}\P_\epsilon g(Y)\right|_{\epsilon=0}
=\P[g(Y)s(Z)].
\end{align*}
These identities account for the perturbations of the integration measures in the calculations below.

\medskip
\noindent\emph{Part (i): the $\xi_\phi$ family.}
By definition,
\[
\alpha_\phi\xi_\phi(\P_\epsilon)
=
\int_{\R}\int_{\cX}
\phi\{Q_{t,\epsilon}(x)-F_{Y,\epsilon}(t)\}
\,\d\P_{X,\epsilon}(x)\d\P_{Y,\epsilon}(t).
\]
Differentiating first with respect to the two integration measures and then with respect to the integrand gives the following four terms:
\begin{align*}
\left.\frac{\d}{\d\epsilon}
\{\alpha_\phi\xi_\phi(\P_\epsilon)\}\right|_{\epsilon=0}
={}&
\P\!\left[s(Z)\int_{\cX}\phi\{\Delta_Y(u)\}
\,\d\P_X(u)\right]+\P\!\left[s(Z)\int_{\R}\phi\{\Delta_t(X)\}
\,\d\P_Y(t)\right]\\
&+\int_{\R}\int_{\cX}\phi'\{\Delta_t(x)\}\dot Q_t(x)
\,\d\P_X(x)\d\P_Y(t)-\int_{\R}\int_{\cX}\phi'\{\Delta_t(x)\}\dot F_Y(t)
\,\d\P_X(x)\d\P_Y(t).
\end{align*}
By the definitions in the theorem, the first two terms are $\P\{A_{\phi,\P}^{\xi}(Y)s(Z)\}$ and $\P\{B_{\phi,\P}^{\xi}(X)s(Z)\}$.  For the third term, substitute the expression for $\dot Q_t(x)$ and apply the tower property:
\begin{align*}
\int_{\R}\int_{\cX}
\phi'\{\Delta_t(x)\}\dot Q_t(x)
\,\d\P_X(x)\d\P_Y(t)
&=
\int_{\R}\E\!\left[
\phi'\{\Delta_t(X)\}
\{\ind(Y\le t)-Q_t(X)\}s(Z)
\right]\d\P_Y(t)\\
&=
\P\!\left[s(Z)
\int_{\R}\phi'\{\Delta_t(X)\}
\{\ind(Y\le t)-Q_t(X)\}\d\P_Y(t)
\right]\\
&=\P\{C_{\phi,\P}^{\xi}(Z)s(Z)\}.
\end{align*}
For the fourth term, use the definition of $\overline\phi_\P(t)$ and the formula for $\dot F_Y(t)$ to obtain
\begin{align*}
-\int_{\R}\int_{\cX}
\phi'\{\Delta_t(x)\}\dot F_Y(t)
\,\d\P_X(x)\d\P_Y(t)&=-\int_{\R}\overline\phi_\P(t)
\P[\{\ind(Y\le t)-F_Y(t)\}s(Z)]\d\P_Y(t)\\
&=\P\!\left[s(Z)
\left\{-\int_{\R}\overline\phi_\P(t)
\{\ind(Y\le t)-F_Y(t)\}\d\P_Y(t)\right\}
\right]\\
&=\P\{D_{\phi,\P}^{\xi}(Y)s(Z)\}.
\end{align*}
Combining the four identities yields
\[
\left.\frac{\d}{\d\epsilon}
\{\alpha_\phi\xi_\phi(\P_\epsilon)\}\right|_{\epsilon=0}
=\P\!\left[s(Z)
\{A_{\phi,\P}^{\xi}(Y)+B_{\phi,\P}^{\xi}(X)
+C_{\phi,\P}^{\xi}(Z)+D_{\phi,\P}^{\xi}(Y)\}
\right].
\]
The centering constants follow from Fubini's theorem and conditional centering:
\begin{align*}
\P A_{\phi,\P}^{\xi}(Y)
&=\int_{\R}\int_{\cX}\phi\{\Delta_t(x)\}
\,\d\P_X(x)\d\P_Y(t)
=\alpha_\phi\xi_\phi(\P),\\
\P B_{\phi,\P}^{\xi}(X)
&=\int_{\cX}\int_{\R}\phi\{\Delta_t(x)\}
\,\d\P_Y(t)\d\P_X(x)
=\alpha_\phi\xi_\phi(\P),\\
\P C_{\phi,\P}^{\xi}(Z)
&=\int_{\R}\P\!\left[
\phi'\{\Delta_t(X)\}
\E[\ind(Y\le t)-Q_t(X)\mid X]
\right]\d\P_Y(t)=0,\\
\P D_{\phi,\P}^{\xi}(Y)
&=-\int_{\R}\overline\phi_\P(t)
\{\P\ind(Y\le t)-F_Y(t)\}\d\P_Y(t)=0.
\end{align*}
Thus, as $\P s=0$, subtracting $2\alpha_\phi\xi_\phi(\P)$ from the expression inside the expectation does not change the derivative. Moreover, Proposition~\ref{prop:normalizing_constants} implies that $\alpha_\phi$ is constant along the submodel. We conclude that
\begin{align*}
\left.\frac{\d}{\d\epsilon}\xi_\phi(\P_\epsilon)
\right|_{\epsilon=0}
=\P\!\left[s(Z)\frac{
A_{\phi,\P}^{\xi}(Y)+B_{\phi,\P}^{\xi}(X)
+C_{\phi,\P}^{\xi}(Z)+D_{\phi,\P}^{\xi}(Y)
-2\alpha_\phi\xi_\phi(\P)}{\alpha_\phi}\right],
\end{align*}
which establishes~\eqref{eq:eif_xi_phi}.

\medskip
\noindent\emph{Part (ii): the $\Lambda_\phi$ family.}
Here we begin with
\[
\beta_\phi\Lambda_\phi(\P_\epsilon)
=
\int_{\R}\int_{\cX^2}
\phi\{Q_{t,\epsilon}(x)-Q_{t,\epsilon}(u)\}
\,\d\P_{X,\epsilon}(x)\d\P_{X,\epsilon}(u)
\d\P_{Y,\epsilon}(t).
\]
Differentiating the threshold measure gives
\begin{align*}
&\P\!\left[s(Z)\int_{\cX^2}
\phi\{Q_Y(u)-Q_Y(v)\}
\d\P_X(u)\d\P_X(v)\right]
=\P\{A_{\phi,\P}^{\Lambda}(Y)s(Z)\}.
\end{align*}
Differentiating the first covariate measure gives
\[
\P\!\left[s(Z)\int_{\R}\int_{\cX}
\phi\{Q_t(X)-Q_t(u)\}\d\P_X(u)\d\P_Y(t)\right],
\]
whereas differentiating the second one gives
\[
\P\!\left[s(Z)\int_{\R}\int_{\cX}
\phi\{Q_t(u)-Q_t(X)\}\d\P_X(u)\d\P_Y(t)\right].
\]
Their sum is precisely $\P\{B_{\phi,\P}^{\Lambda}(X)s(Z)\}$.

It remains to differentiate the two conditional CDFs in the integrand. This yields
\begin{align*}
&\int_{\R}\int_{\cX^2}
\phi'\{Q_t(x)-Q_t(u)\}
\{\dot Q_t(x)-\dot Q_t(u)\}
\d\P_X(x)\d\P_X(u)\d\P_Y(t)\\
=&\int_{\R}\int_{\cX^2}
\phi'\{Q_t(x)-Q_t(u)\}\dot Q_t(x)
\d\P_X(x)\d\P_X(u)\d\P_Y(t)\\
&\qquad-\int_{\R}\int_{\cX^2}
\phi'\{Q_t(x)-Q_t(u)\}\dot Q_t(u)
\d\P_X(x)\d\P_X(u)\d\P_Y(t)\\
=&\int_{\R}\int_{\cX^2}
\Big[\phi'\{Q_t(x)-Q_t(u)\}
-\phi'\{Q_t(u)-Q_t(x)\}\Big]\dot Q_t(x)
\d\P_X(x)\d\P_X(u)\d\P_Y(t).
\end{align*}
The last equality follows by exchanging $x$ and $u$ in the second integral.  Substitution of the conditional-CDF derivative, followed by the tower property, gives
\begin{align*}
&\int_{\R}\int_{\cX^2}
\Big[\phi'\{Q_t(x)-Q_t(u)\}
-\phi'\{Q_t(u)-Q_t(x)\}\Big]\dot Q_t(x)
\d\P_X(x)\d\P_X(u)\d\P_Y(t)\\
=&\P\!\left[s(Z)
\int_{\R}\int_{\cX}
\Big[\phi'\{Q_t(X)-Q_t(u)\}
-\phi'\{Q_t(u)-Q_t(X)\}\Big]
\{\ind(Y\le t)-Q_t(X)\}
\d\P_X(u)\d\P_Y(t)\right]\\
=&\P\{C_{\phi,\P}^{\Lambda}(Z)s(Z)\}.
\end{align*}
Combining these contributions gives
\[
\left.\frac{\d}{\d\epsilon}
\{\beta_\phi\Lambda_\phi(\P_\epsilon)\}\right|_{\epsilon=0}
=\P\!\left[s(Z)
\{A_{\phi,\P}^{\Lambda}(Y)
+B_{\phi,\P}^{\Lambda}(X)
+C_{\phi,\P}^{\Lambda}(Z)\}\right].
\]
To center the resulting expression, note that
\begin{align*}
\P A_{\phi,\P}^{\Lambda}(Y)
&=\int_{\R}\int_{\cX^2}
\phi\{Q_t(x)-Q_t(u)\}
\d\P_X(x)\d\P_X(u)\d\P_Y(t)
=\beta_\phi\Lambda_\phi(\P),\\
\P B_{\phi,\P}^{\Lambda}(X)
&=\int_{\R}\int_{\cX^2}
\big[\phi\{Q_t(x)-Q_t(u)\}
+\phi\{Q_t(u)-Q_t(x)\}\big]
\d\P_X(x)\d\P_X(u)\d\P_Y(t)=2\beta_\phi\Lambda_\phi(\P),\\
\P C_{\phi,\P}^{\Lambda}(Z)
&=\int_{\R}\int_{\cX}
\Big[\phi'\{Q_t(x)-Q_t(u)\}
-\phi'\{Q_t(u)-Q_t(x)\}\Big]\\
&\qquad\times\E[\ind(Y\le t)-Q_t(X)\mid X=x]
\d\P_X(x)\d\P_X(u)\d\P_Y(t)=0.
\end{align*}
Thus, subtracting $3\beta_\phi\Lambda_\phi(\P)$ from the expression inside the expectation leaves the derivative unchanged. Since $\beta_\phi$ is constant along the submodel by Proposition~\ref{prop:normalizing_constants}, we obtain
\begin{align*}
\left.\frac{\d}{\d\epsilon}\Lambda_\phi(\P_\epsilon)
\right|_{\epsilon=0}
=\P\!\left[s(Z)\frac{
A_{\phi,\P}^{\Lambda}(Y)+B_{\phi,\P}^{\Lambda}(X)
+C_{\phi,\P}^{\Lambda}(Z)
-3\beta_\phi\Lambda_\phi(\P)}{\beta_\phi}\right],
\end{align*}
which is~\eqref{eq:eif_lambda_phi}.


\vspace{0.2cm}

It remains to establish the pathwise differentiability
of $\xi_\phi$ and $\Lambda_\phi$ along arbitrary QMD
submodels. Although the preceding calculations were
carried out for bounded scores, we now show that the
derivative identities in~\eqref{eq:eif_xi_phi}
and~\eqref{eq:eif_lambda_phi} continue to hold
without this restriction.

Let $\{\P_\epsilon\}$ be an arbitrary QMD submodel
through $\P$, with score $s\in L_0^2(\P)$, and retain
the notation introduced above. Since $\cX$ and $\R$
are standard Borel spaces, we fix jointly measurable
versions of the conditional distribution functions
$Q_t(x)$ and $Q_{t,\epsilon}(x)$, defined for all
$(x,t)\in\cX\times\R$. These versions may be
specified arbitrarily on sets of zero measure
under their respective $X$-marginals. For convenience, write
\[
\delta_{t,\epsilon}(x)
:=Q_{t,\epsilon}(x)-Q_t(x),
\qquad
\eta_\epsilon(t)
:=F_{Y,\epsilon}(t)-F_Y(t).
\]
By QMD,
\begin{equation}
\label{eq:qmd-tv}
\|\P_\epsilon-\P-\epsilon s\P\|_{\mathrm{TV}}
=o(|\epsilon|).
\end{equation}
In particular,
\[
\|\P_\epsilon-\P\|_{\mathrm{TV}}
=O(|\epsilon|),
\qquad
\sup_{t\in\R}|\eta_\epsilon(t)|=O(|\epsilon|).
\]

We next establish a uniform $L^2$ bound for the
conditional CDFs. Write
\[
H^2(P,Q)
:=
\int(\sqrt{\d P}-\sqrt{\d Q})^2
\]
for the squared Hellinger distance, where the
integral is taken with respect to any common
dominating measure. QMD implies
\[
H^2(\P_\epsilon,\P)
=
\frac{\epsilon^2}{4}\P s^2+o(\epsilon^2)
=O(\epsilon^2).
\]

Fix $t\in\R$, and let $M_t$ and $M_{t,\epsilon}$
denote the distributions of $(X,\ind(Y\le t))$
under $\P$ and $\P_\epsilon$, respectively.
By conditioning on the binary variable
$\ind(Y\le t)$ and expanding the definition
of Hellinger distance, we obtain the identity
\begin{align*}
H^2(M_{t,\epsilon},M_t)=&
H^2(\P_{X,\epsilon},\P_X)+
\int_{\cX}
\Big[
\Big\{\sqrt{Q_{t,\epsilon}(x)}-\sqrt{Q_t(x)}\Big\}^2\\
&\qquad\qquad+
\Big\{\sqrt{1-Q_{t,\epsilon}(x)}
-\sqrt{1-Q_t(x)}\Big\}^2
\Big]
\sqrt{\d\P_X\,\d\P_{X,\epsilon}}.
\end{align*}
Standard algebra gives
\begin{align*}
&\Big\{Q_{t,\epsilon}(x)-Q_t(x)\Big\}^2\le
\Big\{\sqrt{Q_{t,\epsilon}(x)}-\sqrt{Q_t(x)}\Big\}^2
+
\Big\{\sqrt{1-Q_{t,\epsilon}(x)}-
\sqrt{1-Q_t(x)}\Big\}^2.
\end{align*}
Consequently, using
\[
\d\P_X+\d\P_{X,\epsilon}
=
2\sqrt{\d\P_X\,\d\P_{X,\epsilon}}
+
(\sqrt{\d\P_X}-\sqrt{\d\P_{X,\epsilon}})^2,
\]
we obtain
\begin{align*}
&\int_{\cX}
\delta_{t,\epsilon}(x)^2
\,\d(\P_X+\P_{X,\epsilon})(x)\\
\le&
2\int_{\cX}
\Big[
\Big\{\sqrt{Q_{t,\epsilon}(x)}-\sqrt{Q_t(x)}\Big\}^2+
\Big\{\sqrt{1-Q_{t,\epsilon}(x)}
-\sqrt{1-Q_t(x)}\Big\}^2
\Big]
\sqrt{\d\P_X\,\d\P_{X,\epsilon}}\\
&~~~+\int_{\cX}
(\sqrt{\d\P_X}-\sqrt{\d\P_{X,\epsilon}})^2\\
\le&
2H^2(M_{t,\epsilon},M_t)
\le 2H^2(\P_\epsilon,\P),
\end{align*}
where the last step follows from the
data-processing inequality. 
Since the bound is independent of $t$,
QMD implies
\begin{equation}
\label{eq:qmd-Q-bound}
\sup_{t\in\R}
\int_{\cX}\delta_{t,\epsilon}(x)^2
\,\d(\P_X+\P_{X,\epsilon})(x)
=O(\epsilon^2).
\end{equation}
By Cauchy--Schwarz,
\eqref{eq:qmd-Q-bound} further yields
\begin{equation}
\label{eq:qmd-Q-L1}
\sup_{t\in\R}
\int_{\cX}|\delta_{t,\epsilon}(x)|
\,\d\P_{X,\epsilon}(x)
=O(|\epsilon|).
\end{equation}

A further useful identity is that, for every
bounded measurable function $g:\cX\to\R$
and every fixed $t\in\R$,
\begin{equation}
\label{eq:qmd-Q-identity}
\begin{aligned}
&\int_{\cX}g(x)\delta_{t,\epsilon}(x)
\,\d\P_{X,\epsilon}(x)=
(\P_\epsilon-\P)
\bigl[g(X)\{\ind(Y\le t)-Q_t(X)\}\bigr].
\end{aligned}
\end{equation}
Indeed, by the defining property of
$Q_{t,\epsilon}$,
\begin{align*}
&\P_\epsilon
\bigl[g(X)\{\ind(Y\le t)-Q_t(X)\}\bigr]=
\int_{\cX}g(x)
\{Q_{t,\epsilon}(x)-Q_t(x)\}
\,\d\P_{X,\epsilon}(x),
\end{align*}
whereas
\[
\P\bigl[
g(X)\{\ind(Y\le t)-Q_t(X)\}
\bigr]=0
\]
by conditional centering. Thus,
\eqref{eq:qmd-Q-identity} holds exactly
for the fixed versions of the conditional CDFs.

We now control the nonlinear Taylor remainders.
Since $\phi\in C^1([-1,1])$, its derivative is
bounded and uniformly continuous. Let
\[
\omega(h):=
\sup_{\substack{u,v\in[-1,1]\\|u-v|\le h}}
|\phi'(u)-\phi'(v)|.
\]
For $a,a+b\in[-1,1]$, the mean-value
theorem gives
\[
|\phi(a+b)-\phi(a)-\phi'(a)b|
\le |b|\omega(|b|).
\]
For any $\gamma>0$, the right-hand side
is bounded by
\[
\omega(\gamma)|b|
+\frac{2\|\phi'\|_\infty}{\gamma}b^2.
\]

By \eqref{eq:qmd-Q-bound} and the uniform
bound on $\eta_\epsilon$, we have
\begin{align*}
&\int_{\R}\int_{\cX}
\{\delta_{t,\epsilon}(x)-\eta_\epsilon(t)\}^2
\,\d\P_{X,\epsilon}(x)
\d\P_{Y,\epsilon}(t)
=O(\epsilon^2),
\end{align*}
and
\begin{align*}
&\int_{\R}\int_{\cX^2}
\{\delta_{t,\epsilon}(x)
-\delta_{t,\epsilon}(u)\}^2\times
\d\P_{X,\epsilon}(x)
\d\P_{X,\epsilon}(u)
\d\P_{Y,\epsilon}(t)
=O(\epsilon^2).
\end{align*}
Cauchy--Schwarz therefore shows that
the integrated Taylor remainders for both
functionals are bounded by
\[
C|\epsilon|\omega(\gamma)
+C\epsilon^2/\gamma.
\]
Letting $\epsilon\to0$ and then
$\gamma\downarrow0$ establishes that
these remainders are $o(|\epsilon|)$.

\medskip
\noindent\emph{Part (i): the $\xi_\phi$ family.}
A first-order Taylor expansion around
$\Delta_t(x)$ gives
\begin{align}
\alpha_\phi
\{\xi_\phi(\P_\epsilon)-\xi_\phi(\P)\}
&=
\int_{\R}\int_{\cX}
\phi\{\Delta_t(x)\}
\,\d\bigl(
\P_{X,\epsilon}\otimes\P_{Y,\epsilon}
-\P_X\otimes\P_Y
\bigr)(x,t)
\notag\\
&\quad+
\int_{\R}\int_{\cX}
\phi'\{\Delta_t(x)\}
\delta_{t,\epsilon}(x)
\,\d\P_{X,\epsilon}(x)
\d\P_{Y,\epsilon}(t)
\notag\\
&\quad-
\int_{\R}\int_{\cX}
\phi'\{\Delta_t(x)\}
\eta_\epsilon(t)
\,\d\P_{X,\epsilon}(x)
\d\P_{Y,\epsilon}(t)
+o(|\epsilon|).
\label{eq:qmd-xi-expansion}
\end{align}

We analyze the three displayed terms separately.
Since $\phi$ is bounded, the product-measure
expansion gives
\begin{align*}
&\int_{\R}\int_{\cX}
\phi\{\Delta_t(x)\}
\,\d\bigl(
\P_{X,\epsilon}\otimes\P_{Y,\epsilon}
-\P_X\otimes\P_Y
\bigr)(x,t)=
(\P_\epsilon-\P)
\{A_{\phi,\P}^{\xi}(Y)
+B_{\phi,\P}^{\xi}(X)\}
+O(\epsilon^2).
\end{align*}
Here the remainder comes from the product
$(\P_{X,\epsilon}-\P_X)
\otimes(\P_{Y,\epsilon}-\P_Y)$,
whose total-variation norm is $O(\epsilon^2)$.

For the second term in
\eqref{eq:qmd-xi-expansion},
\eqref{eq:qmd-Q-L1} and
$\|\P_{Y,\epsilon}-\P_Y\|_{\mathrm{TV}}
=O(|\epsilon|)$ imply that
\begin{align*}
&\int_{\R}\int_{\cX}
\phi'\{\Delta_t(x)\}
\delta_{t,\epsilon}(x)
\,\d\P_{X,\epsilon}(x)
\d\P_{Y,\epsilon}(t)=
\int_{\R}\int_{\cX}
\phi'\{\Delta_t(x)\}
\delta_{t,\epsilon}(x)
\,\d\P_{X,\epsilon}(x)
\d\P_Y(t)
+O(\epsilon^2).
\end{align*}
Applying \eqref{eq:qmd-Q-identity} with
$g(x)=\phi'\{\Delta_t(x)\}$, for each
fixed $t$, and then integrating over
$\P_Y$ yields
\begin{align*}
&\int_{\R}\int_{\cX}
\phi'\{\Delta_t(x)\}
\delta_{t,\epsilon}(x)
\,\d\P_{X,\epsilon}(x)
\d\P_Y(t)=
(\P_\epsilon-\P)C_{\phi,\P}^{\xi}.
\end{align*}

For the third term, the uniform bound
$\sup_t|\eta_\epsilon(t)|=O(|\epsilon|)$
allows us to replace the perturbed marginal
measures by their values under $\P$,
with an $O(\epsilon^2)$ error. Thus,
\begin{align*}
&-\int_{\R}\int_{\cX}
\phi'\{\Delta_t(x)\}\eta_\epsilon(t)
\,\d\P_{X,\epsilon}(x)
\d\P_{Y,\epsilon}(t)=
-\int_{\R}\overline\phi_\P(t)
\eta_\epsilon(t)\,\d\P_Y(t)
+O(\epsilon^2)=
(\P_\epsilon-\P)D_{\phi,\P}^{\xi}
+O(\epsilon^2).
\end{align*}
Combining these identities with
\eqref{eq:qmd-xi-expansion}, and using
the centering calculations established
above, we obtain
\begin{equation}
\label{eq:qmd-xi-result}
\xi_\phi(\P_\epsilon)-\xi_\phi(\P)
=
(\P_\epsilon-\P)\widetilde\xi_{\phi,\P}
+o(|\epsilon|).
\end{equation}

\medskip
\noindent\emph{Part (ii): the $\Lambda_\phi$ family.}
Similarly, a first-order Taylor expansion gives
\begin{align}
&\beta_\phi
\{\Lambda_\phi(\P_\epsilon)-\Lambda_\phi(\P)\}=
\int_{\R}\int_{\cX^2}
\phi\{Q_t(x)-Q_t(u)\}
\,\d\bigl(
\P_{X,\epsilon}^{\otimes2}\otimes\P_{Y,\epsilon}
-\P_X^{\otimes2}\otimes\P_Y
\bigr)(x,u,t)
\notag\\
&\quad+
\int_{\R}\int_{\cX^2}
\phi'\{Q_t(x)-Q_t(u)\}
\{\delta_{t,\epsilon}(x)
-\delta_{t,\epsilon}(u)\}\times
\d\P_{X,\epsilon}(x)
\d\P_{X,\epsilon}(u)
\d\P_{Y,\epsilon}(t)
+o(|\epsilon|).
\label{eq:qmd-lambda-expansion}
\end{align}
Expanding the product measures, the first term
equals
\begin{align*}
(\P_\epsilon-\P)
\{A_{\phi,\P}^{\Lambda}(Y)
+B_{\phi,\P}^{\Lambda}(X)\}
+O(\epsilon^2).
\end{align*}
For the second term, exchanging $x$ and $u$
in the part involving $\delta_{t,\epsilon}(u)$
gives
\begin{align*}
&\int_{\R}\int_{\cX^2}
\big[
\phi'\{Q_t(x)-Q_t(u)\}-
\phi'\{Q_t(u)-Q_t(x)\}
\big]\delta_{t,\epsilon}(x)\times
\d\P_{X,\epsilon}(x)
\d\P_{X,\epsilon}(u)
\d\P_{Y,\epsilon}(t).
\end{align*}
By \eqref{eq:qmd-Q-L1}, replacing the
integration measures for $u$ and $t$
by $\P_X$ and $\P_Y$, respectively,
introduces an error of order $O(\epsilon^2)$.

Applying \eqref{eq:qmd-Q-identity} with
\[
g(x)=
\int_{\cX}
\big[
\phi'\{Q_t(x)-Q_t(u)\}
-\phi'\{Q_t(u)-Q_t(x)\}
\big]\d\P_X(u),
\]
where $t$ is held fixed, shows that the
second term in
\eqref{eq:qmd-lambda-expansion} equals
\[
(\P_\epsilon-\P)C_{\phi,\P}^{\Lambda}
+O(\epsilon^2).
\]
Combining the two terms and using the
centering calculations already established
gives
\begin{equation}
\label{eq:qmd-lambda-result}
\Lambda_\phi(\P_\epsilon)-\Lambda_\phi(\P)
=
(\P_\epsilon-\P)\widetilde\Lambda_{\phi,\P}
+o(|\epsilon|).
\end{equation}

Finally, the two EIF candidates are bounded.
Thus, applying the total-variation expansion
\eqref{eq:qmd-tv} to
\eqref{eq:qmd-xi-result} and
\eqref{eq:qmd-lambda-result} yields
\begin{align*}
\xi_\phi(\P_\epsilon)-\xi_\phi(\P)
&=
\epsilon\,\P(\widetilde\xi_{\phi,\P}s)
+o(|\epsilon|),\\
\Lambda_\phi(\P_\epsilon)-\Lambda_\phi(\P)
&=
\epsilon\,\P(\widetilde\Lambda_{\phi,\P}s)
+o(|\epsilon|).
\end{align*}
Since these expansions hold along every QMD
submodel, both functionals are pathwise
differentiable at $\P$.
\end{proof}

\begin{proof}[Proof of Corollary~\ref{cor:eif_chatterjee}]

For $\phi(u)=u^2$, using $\P_XQ_t(X)=F_Y(t)$ gives
\begin{align*}
\int_\cX\{Q_t(x)-F_Y(t)\}^2\,\d\P_X(x)
&=\int_\cX Q_t(x)^2\,\d\P_X(x)-F_Y(t)^2,\\
\int_{\cX^2}\{Q_t(x)-Q_t(u)\}^2
\,\d\P_X(x)\d\P_X(u)
&=2\left\{\int_\cX Q_t(x)^2\,\d\P_X(x)-F_Y(t)^2\right\}.
\end{align*}
Since $F_Y(Y)$ is uniform on $[0,1]$, $\int F_Y(t)^2\,\d\P_Y(t)=1/3$. Proposition~\ref{prop:normalizing_constants} therefore yields
\begin{align*}
\xi_\phi(\P)
&=6\left\{\Psi_\P-\frac13\right\}=6\Psi_\P-2,\\
\Lambda_\phi(\P)
&=3\cdot2\left\{\Psi_\P-\frac13\right\}=6\Psi_\P-2.
\end{align*}

Recall that
\[
\Psi_\P
=\int_\R\int_\cX Q_t(x)^2\,\d\P_X(x)\d\P_Y(t).
\]
As in the proof of Theorem~\ref{thm:eif_generalized_chatterjee}, consider $\d\P_\epsilon=(1+\epsilon s)\d\P$ with bounded $s\in L_0^2(\P)$. Differentiating the outer threshold measure contributes $\P[H_\P(Y)s(Z)]$, while differentiating the covariate measure contributes
\[
\P\!\left[s(Z)\int_\R Q_t(X)^2\,\d\P_Y(t)\right].
\]
Differentiating the squared conditional CDF contributes
\begin{align*}
&2\int_\R\int_\cX Q_t(x)\dot Q_t(x)
 \d\P_X(x)\d\P_Y(t)=\P\!\left[
 s(Z)\int_\R 2Q_t(X)
 \{\ind(Y\le t)-Q_t(X)\}\,\d\P_Y(t)\right].
\end{align*}
Combining the three terms and using $\P s=0$ yields the centered IF
\[
H_\P(y)
+\int_\R\{2Q_t(x)\ind(y\le t)-Q_t(x)^2\}\,\d\P_Y(t)
-2\Psi_\P.
\]
Both uncentered contributions have expectation $\Psi_\P$, so the displayed IF is centered.
\end{proof}

\begin{proof}[Proof of Corollary~\ref{prop:vanishing_eif_independence}]
\emph{(i) Independence.} If $X$ and $Y$ are independent, then $Q_t(x)=F_Y(t)$ for $\P_X\otimes\P_Y$-almost every $(x,t)$. Thus $\Delta_t(x)=0$, and $\phi(0)=0$ gives $A_{\phi,\P}^\xi=B_{\phi,\P}^\xi=0$ and $\xi_\phi(\P)=0$. The remaining two terms in~\eqref{eq:eif_xi_phi} satisfy
\begin{align*}
C_{\phi,\P}^\xi(x,y)
&=\phi'(0)\int_\R
 \{\ind(y\le t)-F_Y(t)\}\,\d\P_Y(t),\\
D_{\phi,\P}^\xi(y)
&=-\phi'(0)\int_\R
 \{\ind(y\le t)-F_Y(t)\}\,\d\P_Y(t),
\end{align*}
and cancel pointwise. Hence $\widetilde\xi_{\phi,\P}=0$ almost surely. For the pairwise family, the integrands of $A_{\phi,\P}^\Lambda$ and $B_{\phi,\P}^\Lambda$ are $\phi(0)=0$, while the derivative difference in $C_{\phi,\P}^\Lambda$ is $\phi'(0)-\phi'(0)=0$. Therefore $\widetilde\Lambda_{\phi,\P}=0$ as well.

\smallskip
\noindent\emph{(ii) Perfect functional dependence.}
Suppose $Y=f(X)$ almost surely. For any bounded score $s\in L_0^2(\P)$ and sufficiently small positive or negative $\epsilon$, the perturbed distribution $\d\P_\epsilon=(1+\epsilon s)\d\P$ remains supported on $\{(x,y):y=f(x)\}$. Indeed,
\[
\P_\epsilon\{Y\ne f(X)\}
=\P[\ind\{Y\ne f(X)\}(1+\epsilon s(Z))]=0.
\]
Its marginal distribution of $Y$ remains continuous because $\P_\epsilon\ll\P$, so the normalizing constants remain $\alpha_\phi$ and $\beta_\phi$. Both dependence measures therefore equal one along this entire submodel, by their perfect-dependence characterization:
\[
\xi_\phi(\P_\epsilon)=\Lambda_\phi(\P_\epsilon)=1.
\]
Their derivatives at $\epsilon=0$ vanish for every bounded score. By Theorem~\ref{thm:eif_generalized_chatterjee}, these derivatives are the inner products of the respective EIFs with $s$. Since bounded mean-zero scores are dense in $L_0^2(\P)$ and both EIFs lie in that space, the only possibility is that both EIFs vanish $\P$-almost surely.
\end{proof}

\subsection{Proofs for Section~\ref{sec:estimation}}

\begin{proof}[Proof of Theorem~\ref{thm:general_inefficiency}]

\emph{(i) Chatterjee's rank correlation.} 
Note that
\begin{align*}
&\E_\P\left[
\{\ind(Y\le t)-Q_t(X)\}
\{\ind(Y\le s)-Q_s(X)\}
\mid X
\right]=
\P(Y\le t\wedge s\mid X)-Q_t(X)Q_s(X)
=R_\P(t,s;X)
\end{align*}
and
\begin{align*}
\frac{\widetilde\xi_\P(X,Y)}6
=&H_\P(Y)-\Psi_\P+\int_\R\left[
Q_t(X)^2-\Psi_\P
+2Q_t(X)\{\ind(Y\le t)-Q_t(X)\}
\right]\d\P_Y(t).
\end{align*}
Both terms on the right are centered: the expectation
of $H_\P(Y)$ is $\Psi_\P$, and the integral has mean zero.

By continuity of $F_Y$, the function
$H_\P(t)=\E_\P[Q_t(X)^2]$ is continuous,
nondecreasing, and has limits zero and one.
Integration by parts gives
\[
H_\P(Y)
=
1-\int_\R\ind(Y\le t)\,\d H_\P(t),
\quad \P\text{-almost surely}.
\]
It follows that
\[
\Var_\P\{H_\P(Y)\}
=
\int_{\R^2}C_{22}(t,s)
\,\d H_\P(t)\d H_\P(s).
\]

To compute the remaining covariances, write
\[
L_t
:=
Q_t(X)^2
+2Q_t(X)\{\ind(Y\le t)-Q_t(X)\}.
\]
Conditional centering of the residuals gives
\begin{align*}
{\rm Cov}_\P(L_t,L_s)
=&
{\rm Cov}_\P\{Q_t(X)^2,Q_s(X)^2\}+4\E_\P[Q_t(X)Q_s(X)R_\P(t,s;X)],\\
{\rm Cov}_\P\{L_t,\ind(Y\le s)\}
=&
{\rm Cov}_\P\{Q_t(X)^2,Q_s(X)\}+2\E_\P[Q_t(X)R_\P(t,s;X)].
\end{align*}
For the first identity, the cross-covariances
involving exactly one conditionally centered
residual vanish after conditioning on $X$.
For the second, decompose $\ind(Y\le s)$ into
$Q_s(X)+\{\ind(Y\le s)-Q_s(X)\}$.

Comparison with the definitions of
$C_{11}$ and the corrected $C_{12}$ yields
\begin{align*}
{\rm Cov}_\P(L_t,L_s)
&=
C_{11}(t,s)
-\E_\P[R_\P(t,s;X)^2],~~~{\rm Cov}_\P\{L_t,\ind(Y\le s)\}
=C_{12}(t,s).
\end{align*}

Since $H_\P(Y)$ is represented as the
negative integral of $\ind(Y\le s)$ up to
a constant, expanding the variance of
the displayed EIF gives
\begin{align*}
\frac1{36}\P\widetilde\xi_\P^2
=&
\int_{\R^2}
\left\{
C_{11}(t,s)
-\E_\P[R_\P(t,s;X)^2]
\right\}
\d\P_Y(t)\d\P_Y(s)+
\int_{\R^2}C_{22}(t,s)
\,\d H_\P(t)\d H_\P(s)\\
&-
2\int_{\R^2}C_{12}(t,s)
\,\d\P_Y(t)\d H_\P(s).
\end{align*}
The expression in
Theorem~\ref{thm:rank_graph_variance}(i)
is the same without subtracting
\[
\int_{\R^2}
\E_\P[R_\P(t,s;X)^2]
\,\d\P_Y(t)\d\P_Y(s)
=
\Delta_{1,\P}.
\]
Consequently,
\[
\sigma_{\mathrm{Ch}}^2
=
\P\widetilde\xi_\P^2+36\Delta_{1,\P}.
\]

\smallskip
\noindent\emph{(ii) Chatterjee's graph correlation.}
Let $\widetilde Y,\widetilde Y',\widetilde Y''$ be the conditionally independent copies introduced in the main text, and write $L=F_Y(Y\wedge\widetilde Y)$ for this calculation. Conditional independence gives
\begin{align*}
T_2
&=\E\!\left[
 \E\{L\mid X,Y\}^2\right]
=\E\{m_\P(X,Y)^2\},~~~
T_3
=\E\!\left[
 \E\{L\mid X\}^2\right]
=\E\{\theta_\P(X)^2\}.
\end{align*}
It follows from $\E[m_\P(X,Y)\mid X]=\theta_\P(X)$ that
\begin{align*}
\Delta_{2,\P}=\E[m_\P(X,Y)^2]
 -2\E[m_\P(X,Y)\theta_\P(X)]
 +\E[\theta_\P(X)^2]=T_2-T_3.
\end{align*}
To identify $\Delta_{1,\P}$, use $F_Y(y\wedge y')=\int_\R\ind(y\ge t)\ind(y'\ge t) \,\d\P_Y(t)$. Conditional on $X$, the centered two-variable kernel
\begin{align*}
&L-m_\P(X,Y)-m_\P(X,\widetilde Y)+\theta_\P(X)
\end{align*}
has conditional mean zero given either $Y$ or $\widetilde Y$. Expanding its squared expectation, conditioning on $(X,Y)$ or $(X,\widetilde Y)$ in the mixed terms, gives
\begin{align*}
&\E\big[\{L-m_\P(X,Y)-m_\P(X,\widetilde Y)
 +\theta_\P(X)\}^2\big]=T_1-2T_2+T_3.
\end{align*}
On the other hand, substituting the integral representation of $L$ shows that the expression inside braces equals
\[
\int_\R
\{\ind(Y\ge t)-\P(Y\ge t\mid X)\}
\{\ind(\widetilde Y\ge t)-\P(\widetilde Y\ge t\mid X)\}
\,\d\P_Y(t).
\]
Conditioning on $X$ and using independence of $Y$ and $\widetilde Y$ given $X$, the expectation of its square is
\[
\int_{\R^2}\E\{R_\P(t,s;X)^2\}
\,\d\P_Y(t)\d\P_Y(s)=\Delta_{1,\P}.
\]
Consequently,
\[
T_1-2T_2+T_3=\Delta_{1,\P},
\qquad T_2-T_3=\Delta_{2,\P}.
\]

For the remainder of this argument, write $h(y)=1-2F_Y(y)+H_\P(y)$. Since $Y$ is continuously distributed,
\[
\int_\R\ind(y\le t)\,\d\P_Y(t)=1-F_Y(y),
\qquad
\int_\R Q_t(x)\,\d\P_Y(t)
=1-\E\{F_Y(Y)\mid X=x\}.
\]
Using the definitions of $m_\P$ and $\theta_\P$,
\begin{align*}
m_\P(x,y)
&=\int_\R\ind(y\ge t)\{1-Q_t(x)\}
 \,\d\P_Y(t),~~~\theta_\P(x)
=\int_\R\{1-Q_t(x)\}^2\,\d\P_Y(t).
\end{align*}
Substituting these two formulas and expanding squares gives
\begin{align*}
&h(y)+2m_\P(x,y)-\theta_\P(x)\\
&\quad=1-2F_Y(y)+H_\P(y)
 +2\int_{t\le y}(1-Q_t(x))\,\d\P_Y(t)
 -\int_\R\{1-2Q_t(x)+Q_t(x)^2\}\,\d\P_Y(t)\\
&\quad=H_\P(y)+\int_\R
 \{2Q_t(x)\ind(y\le t)-Q_t(x)^2\}
 \,\d\P_Y(t).
\end{align*}
Thus~\eqref{eq:eif_chatterjee} can be written
\[
\frac{\widetilde\xi_\P(X,Y)}6
=h(Y)+2m_\P(X,Y)-\theta_\P(X)-2\Psi_\P.
\]
Conditioning further gives
\begin{align*}
T_4&=\E\{h(Y)m_\P(X,Y)\},~~~T_5=\E\{h(Y)\theta_\P(X)\},~~~
T_6=\E\{h(Y)^2\}.
\end{align*}
Finally,
\[
\E h(Y)=\E m_\P(X,Y)=\E\theta_\P(X)=\Psi_\P,
\quad
\E[m_\P(X,Y)\theta_\P(X)]=T_3.
\]
Expanding the square of the EIF and collecting terms therefore gives
\begin{align*}
\frac1{36}\P\widetilde\xi_\P^2
&=\E\{h(Y)+2m_\P(X,Y)-\theta_\P(X)-2\Psi_\P\}^2\\
&=T_6+4T_2-3T_3+4T_4-2T_5-4T_7.
\end{align*}
Substituting this equality, together with $T_1-2T_2+T_3=\Delta_{1,\P}$ and $T_2-T_3=\Delta_{2,\P}$, into the graph variance formula yields
\begin{align*}
\frac{\sigma_{\mathrm{AC},d}^2}{36}
&=T_6+4T_2-3T_3+4T_4-2T_5-4T_7+(1+\mathfrak{q}_d)(T_1-2T_2+T_3)
   +\mathfrak{o}_d(T_2-T_3)\\
&=\frac1{36}\P\widetilde\xi_\P^2
 +(1+\mathfrak{q}_d)\Delta_{1,\P}+\mathfrak{o}_d\Delta_{2,\P},
\end{align*}
which establishes part (ii).

\smallskip

\noindent\emph{(iii) Strict positivity away from perfect dependence.}
Let $V:=F_Y(Y)$, and recall that $V\sim\mathrm{Unif}(0,1)$
by continuity of $F_Y$. We first consider $\Delta_{1,\P}$. By the definition of
$R_\P(t,s;X)$ and conditional Fubini,
\begin{align*}
&\int_{\R^2}R_\P(t,s;X)
\,\d\P_Y(t)\d\P_Y(s)=
\Var_\P\left\{
\int_\R\ind(Y\le t)\,\d\P_Y(t)
\;\middle|\;X
\right\}=
\Var_\P\{F_Y(Y)\mid X\}
=\Var_\P(V\mid X).
\end{align*}
Consequently, Jensen's inequality gives
\begin{align*}
\Delta_{1,\P}
&=
\E_\P\left[
\int_{\R^2}R_\P(t,s;X)^2
\,\d\P_Y(t)\d\P_Y(s)
\right]\ge
\E_\P\left[
\left\{
\int_{\R^2}R_\P(t,s;X)
\,\d\P_Y(t)\d\P_Y(s)
\right\}^2
\right]=
\E_\P\left[
\Var_\P(V\mid X)^2
\right].
\end{align*}
Thus, $\Delta_{1,\P}=0$ implies
$\Var_\P(V\mid X)=0$ almost surely.
Conversely, if $\Var_\P(V\mid X)=0$ almost surely,
then $Y$ is conditionally degenerate given $X$,
which implies $R_\P(t,s;X)=0$ almost surely
for every $t,s$. Hence $\Delta_{1,\P}=0$.

We next consider $\Delta_{2,\P}$. For each $x$, define
\[
g_x(v)
:=
\E_\P[\min(v,V)\mid X=x]
=
\int_0^v
\P(V\ge u\mid X=x)\,\d u.
\]
Since $F_Y$ is nondecreasing,
\[
m_\P(X,Y)=g_X(V),
\qquad
\theta_\P(X)=\E_\P[g_X(V)\mid X].
\]
It follows that
\[
\Delta_{2,\P}
=
\E_\P\left[
\Var_\P\{g_X(V)\mid X\}
\right].
\]
For any $v_1<v_2$,
\[
g_x(v_2)-g_x(v_1)
=
\int_{v_1}^{v_2}
\P(V\ge u\mid X=x)\,\d u.
\]
This quantity is strictly positive whenever
$v_1<v_2$ belong to the essential support
of the conditional distribution of $V$ given
$X=x$. Thus, $g_x$ is strictly increasing
on that support. Consequently,
\[
\Var_\P\{g_X(V)\mid X\}=0
\quad\text{if and only if}\quad
\Var_\P(V\mid X)=0
\]
almost surely, and hence
\[
\Delta_{2,\P}=0
\quad\text{if and only if}\quad
\Var_\P(V\mid X)=0,
\quad\P\text{-almost surely}.
\]

Lastly, since $F_Y$ is continuous,
$Y=F_Y^{-1}(V)$ almost surely, where
$F_Y^{-1}$ denotes the generalized inverse.
Therefore,
\[
\Var_\P(V\mid X)=0,
\quad\P\text{-almost surely}
\]
if and only if $Y$ is almost surely a measurable
function of $X$.

Combining the preceding arguments yields
\[
\Delta_{1,\P}=0
\quad\text{if and only if}\quad
\Delta_{2,\P}=0
\quad\text{if and only if}\quad
Y=f(X),\quad\P\text{-almost surely}
\]
for some measurable function $f$.
Equivalently, both $\Delta_{1,\P}$ and
$\Delta_{2,\P}$ are strictly positive
if and only if perfect functional dependence
fails. The variance decompositions in
parts (i)--(ii), together with
$\mathfrak q_d,\mathfrak o_d>0$,
complete the proof.
\end{proof}

\begin{proof}[Proof of Theorem~\ref{thm:general_efficient_estimator}]
Fix a fold $k$ and condition on the training observations outside $\mathcal I_k$. Write $q_t(x)=\widehat Q_{k,t}(x)$, $F(t)=F_Y(t)$, and $\delta_t(x)=q_t(x)-Q_t(x)$. Thus $q$ is fixed when the validation observations are sampled, $0\le q_t(x)\le1$, and $\|\delta\|_{2,\P_X\otimes\P_Y}\le r_n$. In the calculations below, $Z',Z''$ denote independent copies of $Z=(X,Y)$, also independent of the displayed argument $z$. The conditional remainder bounds below are uniform over $[0,1]$-valued functions $q$.

\smallskip
\noindent\emph{(i) $\xi_\phi$ estimators.}
By the definitions of $\widehat N^\xi_{\phi,k}$ and $\widehat C^\xi_{\phi,k}$,
\begin{align*}
\widehat N^\xi_{\phi,k}
+\PP_{n,k}\widehat C^\xi_{\phi,k}
=\frac1{n_k^2}\sum_{i,j\in\mathcal I_k}
\Big[&\phi\{q_{Y_j}(X_i)-\widehat F_k(Y_j)\}+\phi'\{q_{Y_j}(X_i)-\widehat F_k(Y_j)\}
\{\ind(Y_i\le Y_j)-q_{Y_j}(X_i)\}\Big].
\end{align*}
Replacing $\widehat F_k$ by $F$ and the empirical averages by their population counterparts gives
\begin{align*}
&\int_\R\int_\cX
\Big[\phi\{q_t(x)-F(t)\}
+\phi'\{q_t(x)-F(t)\}\{Q_t(x)-q_t(x)\}\Big]
\,\d\P_X(x)\d\P_Y(t).
\end{align*}
For each $(x,t)$, Taylor's theorem yields
\begin{align*}
\phi\{Q_t(x)-F(t)\}
={}&\phi\{q_t(x)-F(t)\}
+\phi'\{q_t(x)-F(t)\}\{Q_t(x)-q_t(x)\}+O\big(\delta_t(x)^2\big),
\end{align*}
where the implicit constant depends only on $\|\phi''\|_\infty$. Therefore the preceding population expression differs from $\alpha_\phi\xi_\phi(\P)$ by at most $C\|\delta\|_{2,\P_X\otimes\P_Y}^2$.

For $z_j=(x_j,y_j)$, introduce the two bounded kernels
\begin{align*}
H_q^\xi(z_1,z_2)
:=&\phi\{q_{y_2}(x_1)-F(y_2)\}+\phi'\{q_{y_2}(x_1)-F(y_2)\}
\{\ind(y_1\le y_2)-q_{y_2}(x_1)\},\\
G_q^\xi(z_1,z_2,z_3)
:={}&-\Big[\phi'\{q_{y_2}(x_1)-F(y_2)\}\\
&\qquad+\phi''\{q_{y_2}(x_1)-F(y_2)\}
\{\ind(y_1\le y_2)-q_{y_2}(x_1)\}\Big]
\{\ind(y_3\le y_2)-F(y_2)\}.
\end{align*}
Since
\[
\widehat F_k(t)-F(t)
=\frac1{n_k}\sum_{\ell\in\mathcal I_k}
\{\ind(Y_\ell\le t)-F(t)\},
\]
Taylor expansion in $\widehat F_k(Y_j)-F(Y_j)$ yields
\begin{align*}
\widehat N^\xi_{\phi,k}
+\PP_{n,k}\widehat C^\xi_{\phi,k}
={}&\frac1{n_k^2}\sum_{i,j\in\mathcal I_k}
H_q^\xi(Z_i,Z_j)+\frac1{n_k^3}\sum_{i,j,\ell\in\mathcal I_k}
G_q^\xi(Z_i,Z_j,Z_\ell)
+O_\P(n_k^{-1}).
\end{align*}
Indeed, the Taylor remainder is bounded by $C\|\widehat F_k-F\|_\infty^2=O_\P(n_k^{-1})$, uniformly in $q$, by the empirical-CDF maximal inequality.

Applying the Hoeffding decompositions to these bounded second- and third-order $V$-statistics gives
\begin{align*}
\widehat N^\xi_{\phi,k}
+\PP_{n,k}\widehat C^\xi_{\phi,k}
={}&\E H_q^\xi(Z,Z')
+(\PP_{n,k}-\P)\pi_q^\xi
+O_\P(n_k^{-1}),
\end{align*}
where $\E G_q^\xi(Z,Z',Z'')=0$ and
\begin{align*}
\pi_q^\xi(z)
:=&\E H_q^\xi(z,Z')
+\E H_q^\xi(Z',z)
+\E G_q^\xi(Z',Z'',z)
-2\E H_q^\xi(Z,Z').
\end{align*}
The first two positional projections of $G_q^\xi$ vanish because its expectation with respect to the third argument is zero. For $z=(x,y)$, the three nonconstant terms are
\begin{align*}
\E H_q^\xi(z,Z')
={}&\int_\R\Big[
\phi\{q_t(x)-F(t)\}
+\phi'\{q_t(x)-F(t)\}
\{\ind(y\le t)-q_t(x)\}\Big] \,\d\P_Y(t),\\
\E H_q^\xi(Z',z)
={}&\int_\cX\Big[
\phi\{q_y(u)-F(y)\}
+\phi'\{q_y(u)-F(y)\}
\{Q_y(u)-q_y(u)\}\Big] \,\d\P_X(u),\\
\E G_q^\xi(Z',Z'',z)
={}&-\int_\R\int_\cX
\Big[\phi'\{q_t(u)-F(t)\}
+\phi''\{q_t(u)-F(t)\}
\{Q_t(u)-q_t(u)\}\Big]\\
&\hspace{5em}\times\{\ind(y\le t)-F(t)\}
\,\d\P_X(u)\d\P_Y(t).
\end{align*}
At $q=Q$, these become, respectively, $B^\xi_{\phi,\P}(x)+C^\xi_{\phi,\P}(x,y)$, $A^\xi_{\phi,\P}(y)$, and $D^\xi_{\phi,\P}(y)$. Moreover, $\E H_Q^\xi(Z,Z')=\alpha_\phi\xi_\phi(\P)$. Hence Theorem~\ref{thm:eif_generalized_chatterjee}(i) gives
\[
\pi_Q^\xi(z)=\alpha_\phi\widetilde\xi_{\phi,\P}(z).
\]
The boundedness of $q,Q$, together with the mean-value theorem and Cauchy--Schwarz, applied separately to the three displayed integrals, yields
\[
\|\pi_q^\xi-\pi_Q^\xi\|_{L^2(\P)}
\le C\|q-Q\|_{2,\P_X\otimes\P_Y}.
\]
Indeed, differences involving a response evaluated at the threshold $y$ are bounded using
\[
\int_\R\Big\{\int_\cX|\delta_y(u)|\,\d\P_X(u)\Big\}^2
\,\d\P_Y(y)
\le\|\delta\|_{2,\P_X\otimes\P_Y}^2;
\]
the remaining terms are handled identically by interchanging the order of integration. Thus, conditional on the training data,
\[
(\PP_{n,k}-\P)(\pi_q^\xi-\pi_Q^\xi)
=O_\P\left(\frac{\|\delta\|_{2,\P_X\otimes\P_Y}}{\sqrt{n_k}}\right).
\]
Combining this bound with previous results proves
\begin{align*}
\frac{\widehat N^\xi_{\phi,k}
+\PP_{n,k}\widehat C^\xi_{\phi,k}}{\alpha_\phi}
-\xi_\phi(\P)
={}&(\PP_{n,k}-\P)\widetilde\xi_{\phi,\P}+O_\P\left(r_n^2+\frac{r_n}{\sqrt n}+\frac1n\right).
\end{align*}

\smallskip
\noindent\emph{(ii) $\Lambda_\phi$ estimator.}
By expanding the three empirical integrations in its definition,
\begin{align*}
\widehat N^\Lambda_{\phi,k}
+\PP_{n,k}\widehat C^\Lambda_{\phi,k}
=\frac1{n_k^3}\sum_{i,j,\ell\in\mathcal I_k}
H_q^\Lambda(Z_i,Z_j,Z_\ell),
\end{align*}
where, for $z_j=(x_j,y_j)$ and $t=y_3$,
\begin{align*}
H_q^\Lambda(z_1,z_2,z_3)
:={}&\phi\{q_t(x_1)-q_t(x_2)\}+\Big[\phi'\{q_t(x_1)-q_t(x_2)\}
-\phi'\{q_t(x_2)-q_t(x_1)\}\Big]
\{\ind(y_1\le t)-q_t(x_1)\}.
\end{align*}
Its expectation is
\begin{align*}
\E H_q^\Lambda(Z,Z',Z'')
=\int_\R\int_{\cX^2}
\Big[&\phi\{q_t(x)-q_t(u)\}+\{\phi'\{q_t(x)-q_t(u)\}
-\phi'\{q_t(u)-q_t(x)\}\}\\
&\qquad\times\{Q_t(x)-q_t(x)\}\Big]
\,\d\P_X(x)\d\P_X(u)\d\P_Y(t).
\end{align*}
To compare this expression with $\beta_\phi\Lambda_\phi(\P)$, Taylor-expand
\begin{align*}
\phi\{Q_t(x)-Q_t(u)\}
={}&\phi\{q_t(x)-q_t(u)\}+\phi'\{q_t(x)-q_t(u)\}
\big[\{Q_t(x)-q_t(x)\}-\{Q_t(u)-q_t(u)\}\big]\\
&+O\big(\delta_t(x)^2+\delta_t(u)^2\big).
\end{align*}
Exchanging $x$ and $u$ in the second part of the linear term gives
\begin{align*}
&\int_{\cX^2}\phi'\{q_t(x)-q_t(u)\}
\big[\{Q_t(x)-q_t(x)\}-\{Q_t(u)-q_t(u)\}\big]
\,\d\P_X(x)\d\P_X(u)\\
&\quad=
\int_{\cX^2}
\{\phi'\{q_t(x)-q_t(u)\}
-\phi'\{q_t(u)-q_t(x)\}\}
\{Q_t(x)-q_t(x)\}
\,\d\P_X(x)\d\P_X(u).
\end{align*}
This equals the correction term in $\E H_q^\Lambda$. Consequently,
\[
\big|\E H_q^\Lambda(Z,Z',Z'')
-\beta_\phi\Lambda_\phi(\P)\big|
\le C\|\delta\|_{2,\P_X\otimes\P_Y}^2.
\]

The Hoeffding decomposition of the bounded third-order kernel yields
\begin{align*}
\widehat N^\Lambda_{\phi,k}
+\PP_{n,k}\widehat C^\Lambda_{\phi,k}
={}&\E H_q^\Lambda(Z,Z',Z'')
+(\PP_{n,k}-\P)\pi_q^\Lambda
+O_\P(n_k^{-1}),
\end{align*}
where
\begin{align*}
\pi_q^\Lambda(z)
={}&\E H_q^\Lambda(z,Z',Z'')
+\E H_q^\Lambda(Z',z,Z'')+\E H_q^\Lambda(Z',Z'',z)
-3\E H_q^\Lambda(Z,Z',Z'').
\end{align*}
Write $z=(x,y)$. At $q=Q$, the first conditional expectation is
\begin{align*}
\E H_Q^\Lambda(z,Z',Z'')
={}&\int_\R\int_\cX
\Big[\phi\{Q_t(x)-Q_t(u)\}\\
&+\{\phi'\{Q_t(x)-Q_t(u)\}
-\phi'\{Q_t(u)-Q_t(x)\}\}
\{\ind(y\le t)-Q_t(x)\}\Big]
\,\d\P_X(u)\d\P_Y(t).
\end{align*}
The other two expectations simplify to
\begin{align*}
\E H_Q^\Lambda(Z',z,Z'')
&=\int_\R\int_\cX
\phi\{Q_t(u)-Q_t(x)\}
\,\d\P_X(u)\d\P_Y(t),\\
\E H_Q^\Lambda(Z',Z'',z)
&=\int_{\cX^2}
\phi\{Q_y(u)-Q_y(v)\}
\,\d\P_X(u)\d\P_X(v),
\end{align*}
because $\E[\ind(Y\le t)-Q_t(X)\mid X]=0$. The sum of these three expressions is $A^\Lambda_{\phi,\P}(y)+B^\Lambda_{\phi,\P}(x) +C^\Lambda_{\phi,\P}(x,y)$. Since $\E H_Q^\Lambda=\beta_\phi\Lambda_\phi(\P)$, Theorem~\ref{thm:eif_generalized_chatterjee}(ii) gives
\[
\pi_Q^\Lambda(z)
=\beta_\phi\widetilde\Lambda_{\phi,\P}(z).
\]
For general $q$, the differences of all three conditional expectations from their values at $Q$ are bounded in $L^2(\P)$ by $C\|\delta\|_{2,\P_X\otimes\P_Y}$. Indeed, $\phi$ and $\phi'$ are Lipschitz on $[-1,1]$, $\ind(Y\le t)-q_t(X)$ is bounded, and each conditional expectation integrates the remaining covariate and threshold arguments against their marginal laws. Thus,
\[
\|\pi_q^\Lambda-\pi_Q^\Lambda\|_{L^2(\P)}
\le C\|\delta\|_{2,\P_X\otimes\P_Y},
\]
and conditional variance control gives
\[
(\PP_{n,k}-\P)(\pi_q^\Lambda-\pi_Q^\Lambda)
=O_\P\left(
\frac{\|\delta\|_{2,\P_X\otimes\P_Y}}{\sqrt{n_k}}
\right).
\]
Combining this with the previous calculations proves
\begin{align*}
\frac{\widehat N^\Lambda_{\phi,k}
+\PP_{n,k}\widehat C^\Lambda_{\phi,k}}{\beta_\phi}
-\Lambda_\phi(\P)
={}&(\PP_{n,k}-\P)\widetilde\Lambda_{\phi,\P}+O_\P\left(r_n^2+\frac{r_n}{\sqrt n}+\frac1n\right).
\end{align*}

Finally, multiply the two foldwise expansions by $n_k/n$ and sum over $k$. Because $K$ is fixed and $n_k\asymp n$,
\[
\sum_{k=1}^K\frac{n_k}{n}(\PP_{n,k}-\P)f
=(\PP_n-\P)f
=\frac1n\sum_{i=1}^n\{f(Z_i)-\P f\}.
\]
Both EIFs have $\P$-mean zero, so this gives \eqref{eq:AL_two_general} for $a=\xi$ and $a=\Lambda$. The rate $r_n=o_\P(n^{-1/4})$ makes the remainder $o_\P(n^{-1/2})$. 
\end{proof}

\begin{proof}[Proof of Corollary~\ref{cor:efficient_xi}]
Fix a fold $k$, condition on its training observations, and write $q_t(x)=\widehat Q_{k,t}(x)$ and $\delta_t(x)=q_t(x)-Q_t(x)$. For $z=(x,y)$ and $z'=(x',y')$, define
\[
H_q(z,z')
:=2q_{y'}(x)\ind(y\le y')-q_{y'}(x)^2.
\]
By~\eqref{eq:xi_eff_double_sum}, the fold-$k$ contribution, before multiplication by $6n_k/n$ and subtraction of $2$, is exactly $n_k^{-2}\sum_{i,j\in\mathcal I_k}H_q(Z_i,Z_j)$. The population expectation has the exact decomposition
\begin{align*}
\E H_q(Z,Z')
&=\int_\R\int_\cX
\{2q_t(x)Q_t(x)-q_t(x)^2\}
\,\d\P_X(x)\d\P_Y(t)\\
&=\int_\R\int_\cX
\{Q_t(x)^2-\delta_t(x)^2\}
\,\d\P_X(x)\d\P_Y(t)\\
&=\Psi_\P-\|\delta\|_{2,\P_X\otimes\P_Y}^2.
\end{align*}
The bounded, generally nonsymmetric $V$-statistic admits the Hoeffding expansion
\begin{align*}
\frac1{n_k^2}\sum_{i,j\in\mathcal I_k}H_q(Z_i,Z_j)
={}&\E H_q(Z,Z')
+(\PP_{n,k}-\P)\pi_q+O_\P(n_k^{-1}),
\end{align*}
where its first-order projection is
\begin{align*}
\pi_q(x,y)
={}&\E H_q((x,y),Z')
+\E H_q(Z',(x,y))-2\E H_q(Z,Z')\\
={}&\int_\R
\{2q_t(x)\ind(y\le t)-q_t(x)^2\}\,\d\P_Y(t)\\
&+\int_\cX
\{2q_y(u)Q_y(u)-q_y(u)^2\}\,\d\P_X(u)
-2\E H_q(Z,Z').
\end{align*}
Substituting $q=Q$ gives
\begin{align*}
\pi_Q(x,y)
={}&\int_\R
\{2Q_t(x)\ind(y\le t)-Q_t(x)^2\}\,\d\P_Y(t)+\int_\cX Q_y(u)^2\,\d\P_X(u)-2\Psi_\P
=\frac16\widetilde\xi_\P(x,y),
\end{align*}
by Corollary~\ref{cor:eif_chatterjee}.

It remains to control the difference between $\pi_q$ and $\pi_Q$. Using $q=Q+\delta$ and $2qQ-q^2=Q^2-\delta^2$, we obtain the exact identity
\begin{align*}
\pi_q(x,y)-\pi_Q(x,y)
={}&\int_\R
\big[2\delta_t(x)\{\ind(y\le t)-Q_t(x)\}
-\delta_t(x)^2\big] \,\d\P_Y(t)\\
&-\int_\cX\delta_y(u)^2\,\d\P_X(u)
+2\|\delta\|_{2,\P_X\otimes\P_Y}^2.
\end{align*}
Since $|\delta_t(x)|\le1$, Jensen's inequality gives
\begin{align*}
&\E\left[
\left\{\int_\R|\delta_t(X)|\,\d\P_Y(t)\right\}^2
\right]
\le\|\delta\|_{2,\P_X\otimes\P_Y}^2\qquad\text{and}\qquad\E\left[
\left\{\int_\cX\delta_Y(u)^2\,\d\P_X(u)\right\}^2
\right]
\le\|\delta\|_{2,\P_X\otimes\P_Y}^2,
\end{align*}
and the same bound holds with $|\delta_t(X)|$ replaced by $\delta_t(X)^2$. Hence
\[
\|\pi_q-\pi_Q\|_{L^2(\P)}
\le C\|\delta\|_{2,\P_X\otimes\P_Y}.
\]
Conditionally on the training sample, therefore,
\[
(\PP_{n,k}-\P)(\pi_q-\pi_Q)
=O_\P\left(
\frac{\|\delta\|_{2,\P_X\otimes\P_Y}}{\sqrt{n_k}}
\right).
\]
Combining this conditional empirical bound with the exact population bias and the bounded-kernel $V$-statistic remainder yields
\begin{align*}
&\frac1{n_k^2}\sum_{i,j\in\mathcal I_k}H_q(Z_i,Z_j)-\Psi_\P=(\PP_{n,k}-\P)\frac{\widetilde\xi_\P}{6}
+O_\P\left(r_n^2+\frac{r_n}{\sqrt n}+\frac1n\right).
\end{align*}
Multiplying by $6n_k/n$, summing over the fixed number of folds, and using $\xi(\P)=6\Psi_\P-2$ give
\[
\widehat\xi_n^{\mathrm{eff}}-\xi(\P)
=\frac1n\sum_{i=1}^n\widetilde\xi_\P(X_i,Y_i)
+O_\P\left(r_n^2+\frac{r_n}{\sqrt n}+\frac1n\right).
\]
Assumption~\ref{ass:efficient_first_stage} makes the remainder $o_\P(n^{-1/2})$. 
\end{proof}

\section*{Acknowledgment of AI use}

The authors used AI during the preparation of this manuscript for language editing, literature exploration and cross-checking, and assistance in reviewing the exposition and technical arguments.  All mathematical results, proofs, interpretations, and cited references were independently checked and verified by the authors, who take full responsibility for the contents of the paper.

\section*{Funding}
The work of HD has been partially supported by the Deutsche Forschungsgemeinschaft (DFG);  TRR 391 {\it Spatio-temporal Statistics for the Transition of Energy and Transport} (520388526); Research unit 5381 \textit{Mathematical Statistics in the Information Age} (460867398). The work of FH has been partially supported by U.S. NSF DMS-2210019, DMS-2514233, and SES-2520486.

{
\bibliographystyle{apalike}
\bibliography{AMS}
}

\end{document}